\documentclass[letterpaper,twocolumn,10pt]{article}

\usepackage[latin1]{inputenc}
\usepackage[english]{babel}
\usepackage{scs}
\usepackage{graphicx}
\usepackage{amssymb,amsmath}
\usepackage{hyperref}   %%%% I included this
\usepackage{caption}       %%%% I included this
\usepackage{subcaption}     %%%% I included this
\usepackage{amsthm}   %%%% I included this

\newtheorem{theorem}{Theorem}[section]
\newtheorem{corollary}[theorem]{Corollary}
\newtheorem{lemma}[theorem]{Lemma}

\newtheorem{algorithm}[theorem]{Algorithm}

\newtheoremstyle{normalnoit}{}{}{}{ }{\bf }{.}{ }{}
\theoremstyle{normalnoit}

\newtheorem{remark}[theorem]{Remark}

\newcommand{\be}{\begin{equation}}
\newcommand{\ee}{\end{equation}}
\newcommand{\bt}{\begin{theorem}}
\newcommand{\et}{\end{theorem}}
\begin{document}

%
% title and abstract
%
\title{Signal Processing based on Stable radix-2 DCT Algorithms having Orthogonal Factors}
\author{
Sirani M. Perera\\
%Daytona State College\\
%\url{pereras@daytonastate.edu} %%%%%%%%%%% INCLUDE %%%%%%%%%%
}

\maketitle


\begin{abstract}
This paper presents stable, radix-2, completely recursive discrete cosine transformation algorithms DCT-I and DCT-III solely based on DCT-I, DCT-II, DCT-III, and DCT-IV having sparse and orthogonal factors. Error bounds for computing the completely recursive DCT-I, DCT-II, DCT-III, and DCT-IV algorithms having sparse and orthogonal factors are addressed. Image compression results are presented based on the recursive 2D DCT-II and DCT-IV algorithms for image size $512 \times 512$ pixels with transfer block sizes $8 \times 8$, $16 \times 16$, and $32 \times 32$ with $93.75\%$ absence of coefficients in each transfer block. Finally signal flow graphs are demonstrated based on the completely recursive DCT-I, DCT-II, DCT-III, and DCT-IV algorithms having orthogonal factors.  
\end{abstract}

%%%%%%% Introduction %%%%%%%%
%%%%%%%%%%%%%%%%%
%%%%%%%%%%%%%%
%%%%%%% Introduction %%%%%%%%
%
%%%%%%%%%%%%%%
%%%%%%% Introduction %%%%%%%%
\section{Introduction}
\label{intro}
The Fast Fourier Transform is used to efficiently compute the Discrete Fourier Transform (DFT) and its inverse. The DFTs are widely used in numerous applications in applied mathematics and electrical engineering \cite{VL92, S86, S99, BYR06, PT05, W84}, etc. 
\newline\newline
%%%%%%%%%%%%%%%%%%%%%%%%%%%%
%%%%%%%%%%%%%%%%%%%%%%%%%%%%%%%%%%%
The DFT uses complex arithmetic. The DFT of a sequence of $n$-input $\{x_k\}_{k=0}^{n-1}$ is the sequence of $n$-output $\{y_k\}_{k=0}^{n-1}$ defined via 
\be
\left[\begin{array}{c}
y_0\\ 
y_1\\ 
\vdots\\ 
y_{n-1}
\end{array}\right]=\frac{1}{\sqrt{n}}\left[\begin{array}{llll}
1 & 1 &  \cdots & 1\\ 
1 & \omega_n  & \cdots & \omega_n^{n-1}\\ 
%1 & \omega_n^2 &  \cdots & \omega_n^{2(n-1)}\\ 
 \vdots&   \vdots &  &\vdots \\ 
 1&  \omega_n^{n-1} & \cdots & \omega_n^{(n-1)(n-1)}
\end{array}\right]
\left[\begin{array}{c}
x_0\\ 
x_1\\ 
\vdots\\ 
x_{n-1}
\end{array}\right]
\label{DFT-I-O}
\ee
where $\omega_n=e^{-\frac{2\pi i}{n}}$. There exist real analogues of the DFT, namely the Discrete Cosine Transforms and Discrete Sine Transforms, the main types are from I to IV. Similar to (\ref{DFT-I-O}), the I-IV variants of cosine and sine matrices transform the sequence of $n$-input into a sequence of $n$-output via the transform matrices stated in Table (\ref{tbl:sct}),  
%%%%%%%%%%%%%%%%%%%%%%%%%%%%%%%%%%%%%%%%%%%%%%%%
\begin{table}[h]
\begin{center}
\begin{tabular}{ l  |  l }
\hline
 Cosine and Sine Transforms & Inverse Transforms \\
\hline
$C_{n+1}^I$ =  $\sqrt{\frac{2}{n}}\left [\epsilon_n(j)\:\epsilon_n(k)\:{\rm cos}\:\frac{jk\pi }{n}  \right ]$ &
$\left[C_{n+1}^{I}\right]^{-1} =C_{n+1}^{I}$\\
\hline
$C_{n}^{II}$ = $\sqrt{\frac{2}{n}}\left [ \epsilon_n(j)\:{\rm cos}\:\frac{j(2k+1)\pi }{2n}  \right ] $ &
$\left[C_{n}^{II}\right]^{-1}=C_{n}^{III}$\\
\hline
$C_{n}^{III}$ = $\sqrt{\frac{2}{n}}\left [ \epsilon_n(k)\:{\rm cos}\:\frac{(2j+1)k\pi }{2n}  \right ]$ &
$\left[C_{n}^{III}\right]^{-1}=C_{n}^{II}$  \\
\hline
$C_{n}^{IV}$ = $\sqrt{\frac{2}{n}}\left [ \:{\rm cos}\:\frac{(2j+1)(2k+1)\pi }{4n}  \right ]$ &
$\left[C_{n}^{IV}\right]^{-1}=C_{n}^{IV}$\\
\hline
$S_{n-1}^I$ =  $\sqrt{\frac{2}{n}}\left [\:{\rm sin}\:\frac{(j+1)(k+1)\pi }{n}  \right ]$ &
$\left[S_{n-1}^I\right]^{-1}=S_{n-1}^I$\\
\hline
$S_{n}^{II}$ = $\sqrt{\frac{2}{n}}\left [ \epsilon_n(j+1)\:{\rm sin}\:\frac{(j+1)(2k+1)\pi }{2n}  \right ]$ & 
$\left[S_{n}^{II}\right]^{-1}=S_{n}^{III}$\\
\hline
$S_{n}^{III}$ = $\sqrt{\frac{2}{n}}\left [ \epsilon_n(k+1)\:{\rm sin}\:\frac{(2j+1)(k+1)\pi }{2n}  \right ]$ & 
$\left[S_{n}^{III}\right]^{-1}=S_{n}^{II}$  \\
\hline
$S_{n}^{IV}$ = $\sqrt{\frac{2}{n}}\left [ \:{\rm sin}\:\frac{(2j+1)(2k+1)\pi }{4n}  \right ]$ &
$\left[S_{n}^{IV}\right]^{-1}=S_{n}^{IV}$\\
\hline
\end{tabular}
\end{center}
\caption{Cosine and Sine Transform Matrices}
\label{tbl:sct}
\end{table}
where for DCT-I $j, k = 0, 1, \cdots, n$, DST-I $j, k = 0, 1, \cdots, n-2$, DCT and DST II-IV $j, k = 0, 1, \cdots, n-1$,  $\epsilon_n(0)=\epsilon_n(n)=\frac{1}{\sqrt{2}}$, $\epsilon_n(j)=1$ for $j \in \{1,2,\cdots,n-1\}$ and $n \geq 2$ is an integer. Among DCT I-IV transformations, $C_{n+1}^I$ was introduced in \cite{WH83}, $C_{n}^{II}$ and its inverse $C_{n}^{III}$ were introduced in \cite{ANR74}, and $C_{n}^{IV}$ was introduced into digital signal processing in \cite{J79}. Moreover, among DST I-IV transformations, $S_{n-1}^I$ and $S_n^{IV}$ were introduced in \cite{J76, J79} and $S_{n}^{II}$ and its inverse $S_{n}^{III}$ were introduced in \cite{KS78}. These classifications were also stated in \cite{W84, PT05}.  
\\\\
%%%%%%%%%%%%%%%%%%%%%%%%%%%%%%%%
%%%%%%%%%%%%%%%%%%%%%%%%%%%%%%%%%%%%
It has been stated, in e.g. \cite{PM03, Sc99, S99}, that these cosine and sine matrices of types I-IV are orthogonal.  Strang, in \cite{S99}, proved that the column vectors of each cosine matrix are eigenvectors of a symmetric second difference matrix under different boundary conditions, and are hence orthogonal. Later Britanak, Yip, and Rao in \cite{BYR06} followed very closely the presentation made by Strang's \cite{S99} to point out that the column vectors of each cosine and sine matrix of types I-VIII are eigenvectors of a symmetric second difference matrix. Due to properties of these DCT and DST, it was shown by many authors (see e.g. \cite{BYR06, B13, CR12, FMW13, HSMR12, JKJ09, KS78, KSN14, KSS14, KR09, LKKP13, MPH12, S99, VZR12, VP09}) that these symmetric and asymmetric (rarely used) versions of DCT and DST can be widely used in image processing, signal processing, finger print enhancement, quick response code (QR code), etc.    
%%%%%%%%%%%%%%%%%%%%%%%%%%%
%%%%%%%%%%%%%%%%%%%%%%%%%%%%%
\newline\newline
To obtain real, fast DCT or DST algorithms one can mainly use a polynomial arithmetic technique or a matrix factorization technique. In the polynomial arithmetic technique (see e.g. \cite{ST91}), components of $C_n\:{\bf x}$ or $S_n\:{\bf x}$ are interpreted as the nodes of a degree $n$ polynomial, and then one applies the divide and conquer technique to reduce the degree of the polynomial. Later it was found (see e.g. \cite{TZ00}) that the polynomial arithmetic technique leads to inferior numerical stability of the DCT and DST algorithms. The matrix factorization technique is the direct factorization of the DCT or DST matrices into the product of sparse matrices (see e.g. \cite{W84, RY80, BYR06, PT05, SO13}). The matrix factorization for DST-I in \cite{RY80} used the results in \cite{CSF77} to decompose DST-I into DCT and DST. Also the decomposition for DCT-II in \cite{W84} is a slightly different version of the result in \cite{CSF77}. Though one can find orthogonal matrix factorizations for DCT and DST in \cite{W84}, the resulting algorithms in \cite{W84} are not completely recursive, and hence do not lead to simple recursive algorithms. Moreover \cite{BYR06} has used the same factorization for DST-II and DST-IV as in \cite{W84}. On the other hand, one can use these \cite{W84, BYR06, S99} results to derive recursive, stable algorithms as stated in \cite{PT05, SO13}. 
%%%%%%%%%%%%%%%%%%%%%%%%%%%%%%%%%%%%%%
%%%%%%%%%%%%%%%%%%%%%%%%%%%%
%%%%%%%%%%%%%%%%%%%%%%%%
\newline\newline
However, \cite{PT05} has offered stable, recursive DCT-II and DCT-IV algorithms, based on DCT-II and DCT-IV. Thus this paper completes the picture and provides completely recursive, stable, radix-2 DCT-I and DCT-III algorithms that are solely defined via DCT I-IV, having sparse and orthogonal factors. The paper also addresses the error bounds on computing completely recursive algorithms for DCT I-IV. Moreover, this paper elaborates image compression (absence of $93.75\%$ coefficients in each transfer block) and signal transform designs based on the completely recursive algorithms based on DCT I-IV. 
\newline\newline
In section \ref{sec:factor} we derive factorizations for DCT-I and DCT-III having orthogonal and sparse matrices, and state completely recursive DCT I-IV algorithms solely defined via DCT I-IV having sparse, orthogonal, and rotation/rotation-reflection matrices. Next, in section \ref{sec:cost}, we present the arithmetic cost of computing these algorithms. In section \ref{Errbdd} we derive error bounds in computing these algorithms and discuss the stability. Finally in sections \ref{sec:IMC} and \ref{sec:SFG} respectively, we demonstrate image compression results and signal flow graphs based on these completely recursive DCT I-IV algorithms. 
%%%%%%%%%%%%%%%%%%%%%%%%%%%%%%%%%%%%
%%%%%%%%%%
%%%%%%%%%%

%%%%%%%%%%%%%%%%%%%%%%%%%%%%%%%%%%%%%%%%%%%%%%%%%%%%%%
%%%%%%%%%%%%%%%%%%%%%%%%%%%%%%%%%%%%%%%%%%%%%%%%%%%%%%
%%%%%%%%%%%%%%%%%%%%%%%%%%%%%%%%%%%%%%
%%%%%%%%%%%%%%%%%%%%%%%%%%%%%%%%%%%%
%%%%%%%%%%%%%%%%%%%%%%%%%%%%%%%%%%%%%%
%%%%%%%%%%%%%% DST %%%%%%%%%%%%%%%%%%%%
%%%%%%%%%%%%%%%%%%%%% Sec 3 %%%%%%%%%%%%%%%
\section{Completely recursive radix-2 DCT algorithms having orthogonal factors}
\label{sec:factor}
This section introduces sparse and orthogonal factorizations for DCT-I and DCT-III matrices. In the meantime, we present completely recursive, radix-2 DCT I-IV algorithms solely defined via DCT I-IV, having sparse, orthogonal, and butterfly matrices. One can observe a variant of the DCT-II and DCT-IV algorithms having almost orthogonal factors in \cite{PT05}.  
\newline\\
The following notations and sparse matrices are used frequently in this paper. Denote an involution matrix $\tilde{I}_n$ by $ \tilde{I}_n\:\textbf{x}= \left[x_{n-1},x_{n-2},\cdots,x_0 \right ]^T$,
a diagonal matrix $D_n$ by
$D_n  \: \textbf{x}={\rm diag} \left(  (-1)^k  \right)_{k=0}^{n-1}\textbf{x},$
an even-odd permutation matrix $P_n$ ($n \geq 3$) by
\be
P_n  \: \textbf{x} =\left \{ \begin{array}{c}
\left[x_0,x_2,\cdots,x_{n-2},x_1,x_3,\cdots,x_{n-1} \right ]^T\: \textrm{even\:}n, \\
\left[x_0,x_2,\cdots,x_{n-1},x_1,x_3,\cdots,x_{n-2} \right ]^T\: \textrm{odd\:}n,
\end{array} \right.
\nonumber
\ee
for any $\textbf{x}=\left[x_j\right ]_{j=0}^{n-1}$, and orthogonal matrices ($n \geq 4$) by 
%%%%%%%%%%%%%%%%%%% H hat  and H%%%%%%%%%%
\[
\breve{H}_{n+1}=\frac{1}{\sqrt{2}}\left[\begin{array}{rcr}
 I_{\frac{n}{2}}&  & \widetilde{I}_{\frac{n}{2}}\\ 
 & \sqrt{2} & 
\\
I_{\frac{n}{2}} &  &-\widetilde{I}_{\frac{n}{2}} \\ 
\end{array}\right], \:\:{H}_{n}=\frac{1}{\sqrt{2}}\left[\begin{array}{lr}
 I_{\frac{n}{2}}  & \widetilde{I}_{\frac{n}{2}}\\ 
\\
 I_{\frac{n}{2}}   &-\widetilde{I}_{\frac{n}{2}} \\ 
\end{array}\right], \]
%%%%%%%%%%%%%%%%%%%%%%%%%%%% H  %%%%%%%%%%
%%%%%%%%%%%%%%%%%%%%%%%%  U  %%%%%%%%%%%
\[U_n=\begin{bmatrix}
1 &  & \\ 
 & \frac{1}{\sqrt{2}}\begin{bmatrix}
I_{\frac{n}{2}-1} & I_{\frac{n}{2}-1}\\ 
 I_{\frac{n}{2}-1}& -I_{\frac{n}{2}-1}
\end{bmatrix} & \\ 
 &  & -1
\end{bmatrix}\begin{bmatrix}
I_{\frac{n}{2}} & \\ 
 & D_{\frac{n}{2}}\widetilde{I}_{\frac{n}{2}}
\end{bmatrix},\]
%%%%%%%%%%%%%%%%%%% V  %%%%%%%%%%%%%%
%%%%%% R %%%%%%%%%%%%%%%%%%%%
\[\begin{aligned}
{R}_n & =\begin{bmatrix}
I_{\frac{n}{2}} & \\ 
 & D_{\frac{n}{2}}
\end{bmatrix}\begin{bmatrix}
{\rm diag}\:C_{\frac{n}{2}} &\left ( {\rm diag}\:S_{\frac{n}{2}} \right ) \widetilde{I}_{\frac{n}{2}}\\ 
-\widetilde{I}_{\frac{n}{2}}\left ( {\rm diag}\:S_{\frac{n}{2}} \right ) &  {\rm diag}\:\left ( \widetilde{I}_{\frac{n}{2}}C_{\frac{n}{2}}  \right )\end{bmatrix}
\end{aligned} \]
%%%%%%%%%%%%%%%%%%% Q  %%%%%%%%%%%%%%%%%
where for $k=0, 1, \cdots, \frac{n}{2}-1$
\[\small
C_{\frac{n}{2}}=\left [ {\rm cos} \frac{(2k+1)\pi}{4n} \right ] \hspace{.1in}{\rm and}\hspace{.1in}
S_{\frac{n}{2}}=\left [ {\rm sin} \frac{(2k+1)\pi}{4n} \right ]. 
\]
%%%%%%%%%%%%%%%%%%%%%%%%%%%%
%%%%%%%%%%%%%%%%%%%%%%%%%%%%%%%%%%%%%
%%%%%%%%%%%%%%%%%%%%%%%%%%%%%%%%%%%%%%%%
%%%%%%%%%%%%%%%%%%%%%%%%%%%%%%%%%%%%%%
%\subsection{Fast and stable algorithms for DCT-I-IV}
%\label{subs:Calgo}
%%%%%%%%%%%%%%%%%%%%%%%%%%
DCT-II and DCT-IV algorithms are the keys for the completely recursive procedure, so for a given vector ${\bf x} \in \mathbb{R}^n$, we present algorithms in order ${\bf y}=C_n^{II}\:{\bf x}$, ${\bf y}=C_n^{IV}\:{\bf x}$, ${\bf y}=C_n^{III}\:{\bf x}$ and ${\bf y}=C_{n+1}^{I}\:{\bf x}$. Following the matrix factorizations for DCT-II and DCT-IV in \cite{PT05}, let us first state recursive DCT-II and DCT-IV having orthogonal factors via algorithms $\bf{(\ref{algo:c2})}$ and $\bf{(\ref{algo:c4})}$, respectively.  
%%%%%%%%%%%%%% Algorithm 1 for s2
%%%%%%%%%%%%%%%%%%%%%%%%%%%%%%%%%%%%%
\begin{algorithm} $\left(\bf{cos2(x, n)} \right)$\\
\label{algo:c2}
Input: $n=2^t(t \geq 1)$, $n_1=\frac{n}{2}$, ${\bf x} \in \mathbb{R}^n$.
\begin{enumerate}
\item{If $n=2$, then} \\
$ {\bf y}:=\frac{1}{\sqrt{2}}\left[\begin{array}{rr}
1 & 1\\
1 & -1
\end{array}\right] {\bf x}.
$
\item{ If $n \geq 4$, then}\\
$
\begin{array}{c}
\begin{aligned}
\hspace{.1in}[u_j ]_{j=0}^{n-1}  :=&\: H_n\:{\bf x},   \\
\hspace{.1in} {\bf z1}  :=& \:{\bf cos2} \left(\left[u_j \right]_{j=0}^{n_1-1}, n_1 \right),   \\
\hspace{.1in} {\bf z2}  :=& \:{\bf cos4} \left( \left[u_j \right]_{j=n_1}^{n-1}, n_1 \right),  \\
\hspace{.1in} {\bf y} :=& \: P_n^T\left({\bf z1}^T, {\bf z2}^T \right)^T.   
\end{aligned}
\end{array}
$
\end{enumerate}
Output: ${\bf y}=C_n^{II}{\bf x}$.
\end{algorithm}
%%%%%%%%%%%%%%%%%%%%%% Algorithm DCT2 for DCT4
%\newpage
%%%%%%%%%%%%%%%%%%%%%%%%%%%%%%%%%%%%%
\begin{algorithm} $\left(\bf{cos4(x, n)} \right)$\\
\label{algo:c4}
Input: $n=2^t(t \geq 1)$, $n_1=\frac{n}{2}$, ${\bf x} \in \mathbb{R}^n$.
\begin{enumerate}
\item{If $n=2$, then} \\
$
{\bf y}:=\left[\begin{array}{rr}
\cos \frac{\pi}{8} & \sin \frac{\pi}{8}\\
\sin \frac{\pi}{8} & -\cos \frac{\pi}{8}
\end{array}\right]{\bf x}.
%\nonumber
$
\item{ If $n \geq 4$, then}\\
$
\begin{array}{c}
\begin{aligned}
\hspace{.1in} [u_j ]_{j=0}^{n-1}  :=&\:  R_n\:{\bf x},   \\
\hspace{.1in} {\bf z1} :=&\: {\bf cos2} \left(\left[u_j \right]_{j=0}^{n_1-1}, n_1 \right),   \\
\hspace{.1in} {\bf z2} := &\:{\bf cos2} \left( \left[u_j \right]_{j=n_1}^{n-1}, n_1 \right), \\ 
\hspace{.1in} {\bf w} := & \: U_n  \left({\bf z1}^T, {\bf z2}^T \right)^T, \\ 
\hspace{.1in} {\bf y} :=  & \: P_n^T {\bf w}.   
\end{aligned}
\end{array}
$
\end{enumerate}
Output: ${\bf y}=C_n^{IV}{\bf x}$.
\end{algorithm}
%%%%%%%%%%%%%%%%%%%%%%%%%%%%%%%%%%  DCT 3
%%%%%%%%%%%%%%%%%%%%%%%%%%%%%%%%%%%%%%%%%%%%%%%%%%%%%
By using the well known transpose property between DCT-II and DCT-III we can state an algorithm for DCT-III via $\bf{(\ref {algo:c3})}$. This algorithm executes recursively with the DCT-II and DCT-IV algorithms. 
%%%%%%%%%%%%%% Algorithm C3
%%%%%%%%%%%%%%%%%%%%%%%%%%%%%%%%%%%%%
%\small
\begin{algorithm} $\left(\bf{cos3(x, n)} \right)$\\
\label{algo:c3}
Input: $n=2^t(t \geq 1)$, $n_1=\frac{n}{2}$, ${\bf x} \in \mathbb{R}^n$.
\begin{enumerate}
\item{If $n=2$, then} \\
$
{\bf y}:=\frac{1}{\sqrt{2}}\left[\begin{array}{rr}
1 & 1\\
1 & -1
\end{array}\right]{\bf x}.
$
\item{ If $n \geq 4$, then}\\
$
\begin{array}{c}
\begin{aligned}
\hspace{.1in} [u_j ]_{j=0}^{n-1}  :=&\: P_n\:{\bf x},   \\
\hspace{.1in} {\bf z1}  :=&\: {\bf cos3} \left(\left[u_j \right]_{j=0}^{n_1-1}, n_1 \right),   \\
\hspace{.1in} {\bf z2}  :=&\: {\bf cos4} \left( \left[u_j \right]_{j=n_1}^{n-1}, n_1 \right),  \\
\hspace{.1in} {\bf y} :=&\:  H_n^T \left({\bf z1}^T, {\bf z2}^T \right)^T. \\
\end{aligned}
\end{array}
$
\end{enumerate}
Output: ${\bf y}=C_n^{III}{\bf x}$.
\end{algorithm}
%%%%%%%%%%%%%%%%%%%%%%%%%%%%%%%%%%%
%%%%%%%%%%%%%%%%%%%%%%%%%%%%%%%%%%%%%%%%%%%%%%%%%%%
%\newline\newline
Before stating the algorithm for DCT-I let us derive a sparse and orthogonal factorization for DCT-I. 
\begin{lemma} 
\label{LC1}
Let $n \geq 4$ be an even integer. The matrix $C^I_{n+1}$ can be factored in the form
\begin{equation}
C_{n+1}^I={P}_{n+1}^T\:\left[\begin{array}{c|c}
C_{\frac{n}{2}+1}^{I} & 0\\
\hline\\
0 & C_{\frac{n}{2}}^{III}
\end{array}\right]\breve{H}_{n+1}.
\label{nc1}
\end{equation}
\end{lemma}
%%%% proof
\begin{proof}
Let's apply ${P}_{n+1}$ to $C_{n+1}^{I}$ to permute rows and then partition the resultant matrix. So
\\
(1,1) block becomes $\sqrt{\frac{2}{n}}\left [ \epsilon_n(2j)\epsilon_n(k){\rm cos}\:\frac{2jk\pi}{n}  \right ]_{j,k=0}^{\frac{n}{2}}$,
%%%%%%%%%%%%%%%%%%%%%
\\
(1,2) block becomes
\be
\begin{aligned}
&\sqrt{\frac{2}{n}}\left [ \epsilon_n(2j)\epsilon_n\left(\frac{n}{2}+k+1\right){\rm cos}\:\frac{j(n+2k+2)\pi}{n}  \right ]_{j,k=0}^{\frac{n}{2}, \frac{n}{2}-1}\\
&=\sqrt{\frac{2}{n}}\left [ \epsilon_n(2j)\epsilon_n\left(\frac{n}{2}+k+1\right){\rm cos}\:\frac{j(n-2k-2)\pi}{n}  \right ]_{j,k=0}^{\frac{n}{2}, \frac{n}{2}-1}, 
\end{aligned}
\nonumber
\ee
%%%%%%%%%%%%%%%%%%%%%%
\\
(2,1) block becomes $\sqrt{\frac{2}{n}}\left [ \epsilon_n(k){\rm cos}\:\frac{(2j+1)k\pi}{n}  \right ]_{j,k=0}^{{\frac{n}{2}-1},{\frac{n}{2}}}$, \\
%%%%%%%%%%%%%%%%%%%%%%%%%%
(2,2) block becomes
\be
\begin{aligned}
&\sqrt{\frac{2}{n}}\left [\epsilon_n\left(\frac{n}{2}+k+1\right){\rm cos}\:\frac{(2j+1)(n+2k+2)\pi}{2n}  \right ]_{j,k=0}^{\frac{n}{2}-1}\\
&= \sqrt{\frac{2}{n}}\left [-\epsilon_n\left(\frac{n}{2}+k+1\right){\rm cos}\:\frac{(2j+1)(n-2k-2)\pi}{2n}  \right ]_{j,k=0}^{\frac{n}{2}-1}.
\end{aligned}
\nonumber
\ee
Hence
\small
\be
\begin{aligned}
{P}_{n+1} C_{n+1}^{I} & = \frac{1}{\sqrt{2}}\left[\begin{array}{c|c}
C_{\frac{n}{2}+1}^{I}\:\begin{bmatrix}
{I}_{\frac{n}{2}} & 0\\
0 & \sqrt{2}
\end{bmatrix} & C_{\frac{n}{2}+1}^{I}\:\begin{bmatrix}
\tilde{I}_{\frac{n}{2}}\\
0 
\end{bmatrix}\\
\hline\\
C_{\frac{n}{2}}^{III}\:\begin{bmatrix}
{I}_{\frac{n}{2}} & 0 
\end{bmatrix} & -C_{\frac{n}{2}}^{III}\:\tilde{I}_{\frac{n}{2}}
\end{array}\right]\\
& =\left[\begin{array}{c|c}
C_{\frac{n}{2}+1}^{I} & 0\\
\hline\\ 
0 & C_{\frac{n}{2}}^{III}
\end{array}\right]\breve{H}_{n+1} 
\end{aligned}
\nonumber
\ee
\end{proof}
\normalsize
Thus an algorithm for DCT-I can be stated via $\bf{(\ref {algo:c1})}$, which executes recursively with DCT II-IV algorithms.  
%%%%%%%%%%%%%%%%%%%%%%%%%%%%%%%%%%%%%
%\small
\begin{algorithm} $\left(\bf{cos1(x, n+1)} \right)$\\
\label{algo:c1}
Input: $n=2^t(t \geq 1)$, $n_1=\frac{n}{2}$, ${\bf x} \in \mathbb{R}^{n+1}$.
\begin{enumerate}
\item{If $n=2$, then} \\
$
{\bf y}:=\frac{1}{2}\left[\begin{array}{rrr}
1 & 1 & 0\\
0 & 0 & \sqrt{2}\\
1 & -1 & 0\\
\end{array}\right]\left[\begin{array}{rrr}
1 & 0 & 1\\
0 & \sqrt{2}& 0\\
1 & 0 & -1\\
\end{array}\right]{\bf x}.
%\nonumber
$
\item{ If $n \geq 4$, then}\\
$
\begin{array}{c}
\begin{aligned}
\hspace{.1in}[u_j ]_{j=0}^{n}  :=&\: \breve{H}_{n+1}\:{\bf x},   \\
\hspace{.1in}{\bf z1}  :=&\: {\bf cos1} \left(\left[u_j \right]_{j=0}^{n_1}, n_1+1 \right),   \\
\hspace{.1in} {\bf z2}  :=&\: {\bf cos3} \left( \left[u_j \right]_{j=n_1+1}^{n}, n_1 \right),  \\
\hspace{.1in} {\bf y} :=&\:  P_{n+1}^T \left({\bf z1}^T, {\bf z2}^T \right)^T. \\
\end{aligned}
\end{array}
$
\end{enumerate}
Output: ${\bf y}=C_{n+1}^{I}{\bf x}$.
\end{algorithm}
%%%%%%%%%%%%%%%%%%%%%%%%%%%%%%%%%%%%%%%%%%
%%%%%%%%%%%%%%%%%%%%%%%%%%%%%%%%%%%%%%%%%%%%

%%%%%%%%%%%%%%%%%%%%%%%%%%%%%%%%%%%%%%%%%%%%%%%%%%%
%%%%%%%%%%%%%%%%%%%%%%%%%%%%%%%%%%%%%%%%%%%%%%%%%%%
%%%%%%%%%%%%%%%%% Cost of DCT and DST I-IV
%%%%%%%%%%%%%%%%%%%%%%%%%%%%%%%%%%%%%%%%%%%%%%%%%%%
%%%%%%%%%% 
\section{Arithmetic cost of computing DCT algorithms}
\label{sec:cost}
We first calculate the arithmetic cost of computing DCT I-IV algorithms. Let's denote the number of additions and multiplications required to compute - say a length $n$ DCT II algorithm: ${\bf y}=C_n^{II}\:{\bf x}$ by $\#a(\textrm{DCT-II}, n)$ and $\#m(\textrm{DCT-II}, n)$. Note that the multiplication of $\pm 1$ and permutations are not counted. Once the cost is computed we show numerical results for the speed improvement factor of these algorithms. 
%%%%%%%%%%%%%%%%%%%%%%
%%%%%%%%%%%%%%%%%%%%%%%%%%%%%%%%%
%%%%%%%%%%%%%%%%%%%%%%%%%%%%
\subsection{Number of additions and multiplications in computing DCT I-IV algorithms}
\label{subsec:amDCT}
Here we calculate the arithmetic cost of computing the DCT I-IV algorithms in order $\bf{(\ref{algo:c2})}$, $\bf{(\ref{algo:c4})}$, $\bf{(\ref{algo:c3})}$ and $\bf{(\ref{algo:c1})}$. The cost of addition in computing DCT-II and DCT-IV algorithms is the same as in \cite{PT05}, but the cost of multiplication is different from \cite{PT05}. The latter is because in this paper, not only DCT-I and DCT-III algorithms but also DCT-II and DCT-IV algorithms have orthogonal factors not almost orthogonal factors. Let us first derive explicitly the number of multiplications required to compute DCT-II and DCT-IV algorithms and then the arithmetic cost of DCT-III and DCT-I algorithms respectively.
%%%%%%%%%%%%%%%%%%%%%%%%%%%%%%%%%%%%%%%%%%%%
%%%%%%%%%%%%%%%%%%%%%%%%%%%%%%%%%%%%%%%%%%%%%%%%%%
%%%%%%%%%%%%%%%%%%%%%%%%%%%%%%%%%%%%%%%%%%%%%%%
%%%%%%%% Cost of DCT-II and 
\begin{lemma}\label{Lc41}
Let $n=2^t\:(t \geq 2)$ be given. Using algorithms $\bf{(\ref{algo:c2})}$ and $\bf{(\ref{algo:c4})}$, the arithmetic cost of computing length $n$ DCT-II algorithm is given by  
\begin{eqnarray}
\#a(\textrm{DCT-II}, n) =& \frac{4}{3}nt-\frac{8}{9}n-\frac{1}{9}(-1)^t+1,
\nonumber \\
\#m(\textrm{DCT-II}, n) =& \frac{5}{3}nt-\frac{10}{9}n+\frac{1}{9}(-1)^t+1,
\label{cc2c4}
\end{eqnarray}
\end{lemma}
%%%%%%%%%%%%%%%%
%%%%%%%
\begin{proof}
Following algorithms $\bf{(\ref{algo:c2})}$ and $\bf{(\ref{algo:c4})}$
\begin{eqnarray}
\#m(\textrm{DCT-II}, n)  =& \#m\left(\textrm{DCT-II}, \frac{n}{2} \right) + \#m\left(\textrm{DCT-IV}, \frac{n}{2} \right) 
\nonumber \\ 
& \hspace{.01in} +\: \#m \left({H}_n \right),
\nonumber \\
\#m\:(\textrm{DCT-IV}, n) =& \#m\left(U_n \right) + 2 \cdot \#m\left(\textrm{DCT-II}, \frac{n}{2} \right) 
\nonumber \\ 
& \hspace{.01in} +\: \#m\left(R_n \right).
\nonumber\\
\label{ac24}
\end{eqnarray}
By referring to the structures of ${H}_n$, $U_n$, and $R_n$ 
\begin{eqnarray}
\#a \left({H}_n \right) = n, \:\: \#m \left({H}_n \right) = n,
\nonumber\\
\#a \left({U}_n \right) = n-2, \:\: \#m \left({U}_n \right) = n-2,
\nonumber\\
\#a \left({R}_n \right) = n, \:\: \#m \left({R}_n \right) = 2n,
\nonumber\\
\label{ruc24}
\end{eqnarray}
Thus
\begin{eqnarray}
\#m(\textrm{DCT-II}, n)  = & \#m\left(\textrm{DCT-II}, \frac{n}{2} \right) + 2 \cdot \#m\left(\textrm{DCT-II}, \frac{n}{4} \right) 
\nonumber \\ 
& \hspace{.01in} +\: \frac{5}{2}n -2.
\nonumber
\end{eqnarray}
Since $n=2^t$ we can obtain the linear difference equation of order 2 with respect to $t$
\begin{eqnarray}
\#m(\textrm{DCT-II}, 2^t)  &- &  \#m\left(\textrm{DCT-II}, 2^{t-1} \right) - 2\cdot \#m \left(\textrm{DCT-II}, 2^{t-2} \right) 
\nonumber \\ 
& = & 
 5 \cdot 2^{t-1} -2.
%\label{ac2}
\nonumber
\end{eqnarray}
If $\#m(\textrm{DCT-II}, 2^t)=\alpha^t$(where $\alpha \neq 0$) is a solution then the above follows
\be
{\alpha}^t  - {\alpha}^{t-1} - 2\:({\alpha}^{t-2}) = 5 \cdot 2^{t-1} -2.
\label{eaca}
\ee
The homogeneous solution of the above is given by solving the characteristic equation
\[
{\alpha}^{t-2}({\alpha}^2  - {\alpha} - 2) = 0.
\]
From which we get 
\[
\#m(\textrm{DCT-II}, 2^t)=r_12^t + r_2(-1)^t + {\rm particular\: solution}
\]
where $r_1$ and $r_2$ are constants. Let $\alpha^t=r_3+r_4t\cdot2^t$ (where $r_3$ and $r_4$ are constants) be the particular solution. Substituting this potential equation into (\ref{eaca}) and equating the coefficients we can find that
\[
\#m(\textrm{DCT-II}, 2^t)=r_12^t + r_2(-1)^t + \frac{5}{3}\cdot t\cdot2^t + 1 
\]
Using the initial conditions $\#m \left(\textrm{DCT-II}, 2 \right) = 2$ and $\#m \left(\textrm{DCT-II}, 4 \right) = 10$, we can determine the general solution
\be
\#m(\textrm{DCT-II}, 2^t)=\frac{5}{3}\cdot t\cdot2^t - \frac{10}{9} 2^t + \frac{1}{9} (-1)^t+1  
\label{acansc2}
\ee
Thus substituting $n=2^t$ we can obtain the number of multiplications required to compute DCT-II algorithm as stated in (\ref{cc2c4}). 
\newline
%%%%%%%%%%%%%%%%%%% Mutiplication
Again by algorithms $\bf{(\ref{algo:c2})}$ and $\bf{(\ref{algo:c4})}$ together with (\ref{ruc24}), we can state
\[
\#a(\textrm{DCT-II}, n)  = \#a\left(\textrm{DCT-II}, \frac{n}{2} \right) + 2 \cdot \#a\left(\textrm{DCT-II}, \frac{n}{4} \right) + 2n -2.
\]
Since $n=2^t$, the second order linear difference equation with respect to $t$ can be given via
\begin{eqnarray}
\#a(\textrm{DCT-II}, 2^t)  &-& \#a\left(\textrm{DCT-II}, 2^{t-1} \right) - 2\cdot \#a \left(\textrm{DCT-II}, 2^{t-2} \right) 
\nonumber \\
&=&  2^{t+1} -2.
\nonumber
\end{eqnarray}
As derived analogously in the cost of multiplication, we can solve the above equation under the initial conditions $\#a \left(\textrm{DCT-II}, 2 \right) = 2$ and $\#a \left(\textrm{DCT-II}, 4 \right) = 8$ to obtain 
\be
\#a(\textrm{DCT-II}, n) = \frac{4}{3}nt-\frac{8}{9}n-\frac{1}{9}(-1)^t+1.
\label{mcansc2}
\ee
\end{proof}
%%%%%%%%%%%%%%%%%%%%%%%%%%%%%%%%%%%%%%%%%%%
%%%%%%%%%%%%%%%%%%%%%%%%%%%%%%%%%%%%%%%%%%%
%%%%%%%%%%%%%%%%%%%%%%%%%%%%%%%%%%%%%%
%%%%%%%%%  Cost of DCT-IV
%%%%%%%%%%%%%%%%
%%%%%%%
%%%%%%%%%%%%%%%%%%%%%%%%%%%%%%%%%%%%%%
%%%%%%%%%  Cost of DCT-IV

\begin{corollary}\label{Lco1-41}
Let $n=2^t\:(t \geq 2)$ be given. Using algorithms $\bf{(\ref{algo:c4})}$ and $\bf{(\ref{algo:c2})}$, the arithmetic cost of computing length $n$ DCT-IV algorithm is given by  
\begin{eqnarray}
\#a(\textrm{DCT-IV}, n) &=& \frac{4}{3}nt-\frac{2}{9}n+\frac{2}{9}(-1)^t,
\nonumber \\
\#m(\textrm{DCT-IV}, n) &=& \frac{5}{3}nt+\frac{2}{9}n-\frac{2}{9}(-1)^t.
\label{cc4c4}
\end{eqnarray}
\end{corollary}
%%%%%%%%%%%%%%%%
\begin{proof}
The number of multiplications required to compute DCT-IV algorithm can be found by substituting (\ref{acansc2}) at $\frac{n}{2}(=2^{t-1})$ into the equation (\ref{ac24})
\begin{eqnarray}
\#m(\textrm{DCT-IV}, n)=n-2 &+& 2 \Bigg( \frac{5}{3}\cdot \frac{n}{2}(t-1)- \frac{10}{9}\cdot \frac{n}{2} 
\nonumber \\
&+ &\frac{1}{9}(-1)^{t-1} + 1 \Bigg)+ 2n 
\nonumber
\end{eqnarray}
Simplifying the above gives the cost of multiplication  
\[
\#m(\textrm{DCT-IV}, n)=\frac{5}{3}nt+\frac{2}{9}n-\frac{2}{9}(-1)^t.
\]
%%%%%%%%%%%%%%%%%%%%%%%%%%%%%%%
Similarly, the number of additions required to compute DCT-IV algorithm can be found by substituting (\ref{mcansc2}) at $\frac{n}{2}(=2^{t-1})$ to  
\begin{eqnarray}
\#a\:(\textrm{DCT-IV}, n) &=&\#a\left(U_n \right) + 2 \cdot \#a\left(\textrm{DCT-II}, \frac{n}{2} \right) + \#a\left(R_n \right)
\nonumber \\
 &=& 2 \cdot \#a\left(\textrm{DCT-II}, \frac{n}{2} \right) + 2n - 2.
\nonumber
\end{eqnarray}
Simplifying the above yields
\[
\#a(\textrm{DCT-IV}, n) = \frac{4}{3}nt-\frac{2}{9}n+\frac{2}{9}(-1)^t.
\] 
\end{proof}
%%%%%%%%%%%%%%%%%%%%%%%%%%%%%%%%%%%%%%%%%%%
%%%%%%%%%%%%%%%%%%%%%%%%%%%%%%%%%%%%%%%%%%%
%%%%%%%%%%%%%%%%%%%%%%%%%%%%%%%%%%%%%%
%%%%%%%%%%%%%%
%%%%%%%%%%%%%%%%%%%%%%%%%%%%%%%%%%%%%%%%%%%
%%%%%%%%%%%%%%%%%%%%%%%%%%%%%%%%%%%%%%
%%%%%%%%%%%%%%%%%%%%%%%%%%%%%%%%%%%%%%%%%%%
%%%%%%%%%%%%%%%%%%%%%%%%%%%%%%%%%%%%%%%%%%%
%%%%%%%%%%%%%%%%%%%%%%%%%%%%%%%%%%%%%%
%%%%%%%%%  Cost of DCT-III
\noindent The DCT-III algorithm $\bf{(\ref{algo:c3})}$ was stated using the transpose property of matrices so the following corollary is trivial.   
%%%%%%%%%%%%%%%%%%%%%%%%%%%%%%%%%%%%%%%%
\begin{corollary}\label{Lco2-41}
Let $n=2^t\:(t \geq 2)$ be given. If DCT-III could be computed by using algorithms $\bf{(\ref{algo:c3})}$, $\bf{(\ref{algo:c4})}$, and $\bf{(\ref{algo:c2})}$ then the arithmetic cost of computing a length $n$ DCT-III algorithm is given by  
\begin{eqnarray}
\#a(\textrm{DCT-III}, n)= &\frac{4}{3}nt-\frac{8}{9}n-\frac{1}{9}(-1)^t+1,
\nonumber \\
\#m(\textrm{DCT-III}, n) = & \frac{5}{3}nt-\frac{10}{9}n+\frac{1}{9}(-1)^t + 1.
\label{cc3c3}
\end{eqnarray}
\end{corollary}
%%%%%%%%%%%%
%%%%%%%%%%%%%%%%%%%%%%%%%%%%%%%%%%%%%%%%%%%
%%%%%%%%%%%%%%%%%%%%%%%%%%%%%%%%%%%%%%%%%%%
\begin{remark}
\label{c3usc24}
By using the DCT-III algorithm $\bf{(\ref{algo:c3})}$ and the arithmetic cost of computing the DCT-IV algorithm (in corollary (\ref{Lco1-41})), one can obtain the same results as in corollary (\ref{Lco2-41}). 
\end{remark}
%%%%%%%%%%%%%%%%%%%%%%%%%%%%%%%%%%%%%%%%%%%%%%
\noindent Let us state the arithmetic cost of computing the DCT-I algorithm $\bf{(\ref{algo:c1})}$. 
%%%%%%%%%%%%%%%

%${\color{red}Check the math}$

%%%%%%%%%%%%%%%%%%%% Lemma 4.2
%%%% lemma 4.2
\begin{lemma} \label{Lc42}
Let $n=2^t\:(t \geq 2)$ be given. Using algorithms $\bf{(\ref{algo:c1})}$, $\bf{(\ref{algo:c3})}$, $\bf{(\ref{algo:c4})}$ and $\bf{(\ref{algo:c2})}$, the arithmetic cost of a DCT-I algorithm of length $n+1$ is given by  
\begin{eqnarray}
\#a\:(\textrm{DCT-I}, n+1) =& \frac{4}{3}nt-\frac{14}{9}n+\frac{1}{18}(-1)^t +t + \frac{7}{2} 
\nonumber\\
\#m\:(\textrm{DCT-I}, n+1) =& \frac{5}{3}nt-\frac{22}{9}n-\frac{1}{18}(-1)^t+t + \frac{11}{2} 
\label{cc1c3}
\end{eqnarray}
\end{lemma}

%%%%%%%
\begin{proof}
Referring to the DCT-I algorithm $\bf{(\ref{algo:c1})}$
\be
\begin{aligned}
\#a\:(\textrm{DCT-I}, n+1) = & \#a\: \left(\textrm{DCT-I}, \frac{n}{2}+1 \right) + \#a\: \left(\textrm{DCT-III}, \frac{n}{2} \right) 
\\ &+ \#a\: \left(\breve{H}_{n+1} \right)
\end{aligned}
\label{ac13}
\ee
The structure of $\breve{H}_{n+1}$ leads to  $\:\: \#a \left(\breve{H}_{n+1} \right) = n.$
This together with the arithmetic cost of computing DCT-III (\ref{cc3c3}) algorithm, we can rewrite (\ref{ac13})
\be
\begin{matrix}
\begin{aligned}
\#a\:(\textrm{DCT-I}, n+1)  =& \#a\: \left(\textrm{DCT-I}, \frac{n}{2}+1 \right) + \frac{2}{3}nt- \frac{1}{9}n
\\&+\frac{1}{9}(-1)^{t}+1\\
\end{aligned}
\end{matrix}
\nonumber
\ee
Since $n=2^t$, the above simplifies to the first order linear difference equation (respect to $t \geq 2$)
\be
%\begin{matrix}
\begin{aligned}
\#a\:(\textrm{DCT-I}, 2^t+1)  - & \#a\: \left(\textrm{DCT-I}, 2^{t-1}+1 \right)  
\\ & = \frac{2}{3}t\cdot2^t-\frac{1}{9}2^t +\frac{1}{9}(-1)^t+1
\end{aligned}
%\end{matrix}
\label{eaca1}
\ee
We can obtain the number of additions required to compute the DCT-I algorithm by solving (\ref{eaca1}) under the initial condition $\#a\: \left(\textrm{DCT-I}, 3\right)= 4$. Analogously, one can solve the first order linear difference equation under the initial condition $\#m\: \left(\textrm{DCT-I}, 3\right)= 5$ to obtain the number of multiplications.
\end{proof}
%%%%%%%%%%%%%%%%%%%%%%%%%%%%%%%%%%
%%%%%%%%%%%%%%%%%%%%%%%%%%%%%%%
%%%%%%%%%%%%%%%%%%%%%%%%%%%%%%%%%%%
%%%%%%%%%%%%%%%%%%%%%%%%%%%%%%
%%%%%%%%%%%%%%% Numerical data
%%%%%%%%%%%%%%%%%%%%%%%%%
\subsection{Speed improvement factor of DCT I-IV algorithms}
\label{subsec:NuDCT}
Based on the results in lemmas \ref{Lc41}, \ref{Lc42} and corollaries \ref{Lco1-41}, \ref{Lco2-41}, we graph the speed improvement factor of DCT I-IV algorithms having orthogonal factors. It is known that the speed improvement factor plays a critical role in the DFT algorithms as it gives us an idea about the processing speed of the algorithms. We should recall here that this factor increases with the size of matrix. 
\newline\newline
\indent In our case, the speed improvement factor says the ratio between the number of additions and multiplications required to compute the DCT I-IV algorithms, and the direct computation cost of computing these algorithms which is $2n^2-n$ for DCT II-IV, and $2n^2+3n+1$ for DCT-I. Figure \ref{figSIF} shows the speed improvement factor corresponding to the DCT I-IV algorithms with respect to the size of matrix. These numerical data correspond to MATLAB (R2014a version) with machine precision $2.2 \times 10^{-16}$.
%%%%%%%%%%%%%%%%%%%%%%%%%%%
\begin{figure}[h]
\center
\includegraphics[width=2.5in,height=2in]{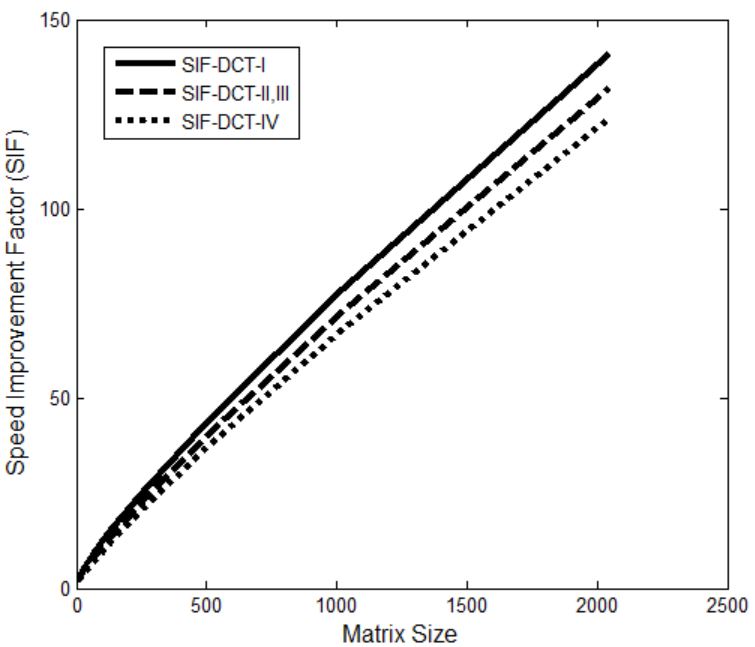}
\caption{Speed improvement factor of DCT I-IV algorithms}
\label{figSIF} 
\end{figure}
%%%%%%%%%%%%%%%%%%%%%%%%%%%%%%%%%%%%%%
%%%%%%%%%%%%%%%%%%%%%%%%%%%%%%%%%%%%%%%%%%%%%%
%%%%%%%%%%%%% Sec 6

%%%%%%%%%%%%%%%%%%%%%%%%%%%%%%%%%%%%%%%%%%
%%%%%%%%%%%%%%%%%%%%%%%%%%%%%%%%%%%%%%%%%%%%%%%%
%%%%%%%%%%%%%%%%%%%%%%%%%%%%%%%%%%%%%%%%%%%%%%
%%%%%%%%%%%%%%%%%%%%%%%%%%%%%%%%%%%%%%%%%%%%%%
%%%%%%%%%%%%%%%%%%%%%%%%%%%%%%%%%%%%%%%%%%%%%
%%%%%%%%%%%%%%%%%%%%%%%%%%%%%%%%%%%%%%%%%
\section{Error bounds and stability of DCT algorithms}
\label{Errbdd}
Error bounds and stability of computing the DCT I-IV algorithms are the main concern in this section. Here, to verify the stability, we will use error bounds (using perturbation of the product of matrices stated in \cite{H61}) in computing these algorithms. Let us assume that the computed trigonometry functions ($d_r:={\rm sin} \:\frac{r\pi}{4n}$ or ${\rm cos}\:\frac{r\pi}{4n}$ are the entries of the butterfly matrix) $\widehat{d}_r$ are used and satisfy
\be
\widehat{d}_r = d_r + \epsilon_r, \:\:\: |\epsilon_r| \leq \mu
\label{scerror}
\ee
for all $r=1,3,5,\cdots,n-1$, where $\mu:=\emph{O}(u)$ and $u$ is the unit roundoff.
\\\\
Let's recall the perturbation of the product of matrices stated in \cite{H61} i.e. if $A_k+\Delta A_k \in \mathbb{R}^{n \times n}$ satisfies $| \Delta A_k | \leq \delta_k |A_k|$ for all $k$, then 
\be
\begin{matrix}
\Bigg| \displaystyle\prod_{k=0}^m \left( A_k+\Delta A_k\right) - \displaystyle\prod_{k=0}^m A_k   \Bigg| 
 \leq \Bigg( \displaystyle\prod_{k=0}^m (1+\delta_k) -1 \Bigg) \displaystyle\prod_{k=0}^m \Bigg| A_k \Bigg|
\end{matrix}
\nonumber
\ee
where $|\delta_k| < u$. Moreover, recall $\displaystyle\prod_{k=1}^n (1+\delta_k)^{\pm 1} = 1+\theta_n$ where $|\theta_n| \leq \frac{nu}{1-nu}=:\gamma_n$ and $\gamma_k+u \leq \gamma_{k+1}$,  $\gamma_k+\gamma_j+\gamma_k\gamma_j \leq \gamma_{k+j}$ from \cite{H61}.
\newline\newline
%%%%%%%%%%%%%%%%%%%%%%%%%%%%%%%%%%%%%%%%%%%%%%%%%%%
%%%%%%%%%%%%%%%%%%%%%%%%%
%%%%%%%%%%%%%%%%%%%%%%%
%%%%%%%%%%%%%%%%%%%%%%%%%%%%%%%%%%%%
%%%%%%%%%%%%%%%%%%%%%%%%%%%%%%%%
Let us derive error bounds for computing recursive DCT I-IV algorithms with the help of the perturbations in a matrix product.
%%%%%%%%%%%%%%%%%%%%%%%%%%%%%%%
%%% Theorm 6.14
\bt \label{ThErrC2}
Let $\widehat{{\bf y}}=fl(C_{n}^{II} {\bf x})$, where $n=2^t(t \geq 2)$, be computed using the algorithms $\bf{(\ref{algo:c2})}$, $\bf{(\ref{algo:c4})}$, and assume that (\ref{scerror}) holds, then
\be 
\frac{\left \| {\bf y}-\widehat{{\bf y}} \right \|_2}{\left \| {\bf y} \right \|_2}\leq \frac{\gamma_7(t-1)}{1-\gamma_7(t-1)}.
\label{Ec2}
\ee
\et
%%%%%%%%%%%%%%%%%%%%%%%%%%%
\begin{proof}
Using the algorithms $\bf{(\ref{algo:c2})}$, $\bf{(\ref{algo:c4})}$, and the computed matrices $\widehat{\textbf{G}}_{k}$ (in terms of the computed $\widehat{d}_r$) for $k=1,2,\cdots,t-2$:
\small
\be
\begin{aligned}
\widehat{{\bf y}}&=fl\left({\textbf{P}}^T_0\: {\textbf{P}}^T_1{\textbf{F}}_1{\textbf{P}}^T_2{\textbf{F}}_2\cdots{\textbf{P}}^T_{t-2}{\textbf{F}}_{t-2}\:{\textbf{C}}_{t-1}\:\widehat{\textbf{G}}_{t-2}\cdots \widehat{\textbf{G}}_{2}\widehat{\textbf{G}}_{1}{\textbf{G}}_{0}\: {\bf x}\right)\\
&={\textbf{P}}^T_0{\textbf{P}}^T_1\left({\textbf{F}}_1 + \Delta{\textbf{F}_1}\right)\cdots{\textbf{P}}^T_{t-2}\left({\textbf{F}}_{t-2} + \Delta{\textbf{F}_{t-2}}\right) ({\textbf{C}}_{t-1} + \Delta{{\textbf{C}}_{t-1}})
\\& \hspace{.4in}
(\widehat{\textbf{G}}_{t-2}+\Delta{\widehat{\textbf{G}}_{t-2}})\cdots( \widehat{\textbf{G}}_{2}+\Delta{\widehat{\textbf{G}}_{2}})
( \widehat{\textbf{G}}_{1}+\Delta{\widehat{\textbf{G}}_{1}})( {\textbf{G}}_{0}+\Delta{{\textbf{G}}_{0}}) {\bf x}
\end{aligned}
%\label{1Es1}
\nonumber
\ee
Each $\textbf{F}_k$ is formed containing a combination of matrices ${I}_{\frac{n}{2^k}}$ and $U_{\frac{n}{2^k}}$. Using the fact that each row in $\textbf{F}_k$ has at most two non-zero entries with mostly ones per row:
\be 
\begin{matrix}
\left | \Delta{{\textbf{F}}_k} \right | \leq {\gamma}_{2}\left | \textbf{F}_k \right |\:\:\:{\rm for}\:\:\:k=1,2,\cdots,t-2
\end{matrix}
%\label{rEs1}
\nonumber
\ee
Also each ${\textbf{G}}_k$ is formed containing a combination of matrices ${H}_{\frac{n}{2^k}}$ and $R_{\frac{n}{2^k}}$ except $\textbf{G}_0={H}_{n}$. Using the fact that each row in ${\textbf{G}_k}$ has at most two non-zero entries per row:
\be 
\begin{matrix}
\left | \Delta{{\textbf{G}}_0} \right | \leq {\gamma}_{2}\:\left | {\textbf{G}}_0 \right |,\:\:\: 
\left | \Delta{\widehat{\textbf{G}}_k} \right | \leq {\gamma}_{3}\:\left | \widehat{\textbf{G}}_k \right |,\\
{\rm for}\:\:\:k=1,2\cdots,t-2
\end{matrix}
%\label{tEs1}
\nonumber
\ee
${\textbf{C}}_{t-1}$ is a block diagonal matrix containing $C_{2}^{II}$ and $C_{2}^{IV}$ hence
\be 
\left | \Delta{{\textbf{C}}_{t-1}} \right | \leq {\gamma}_{3}\:\left | {\textbf{C}}_{t-1} \right |
%\label{ssEs1}
\nonumber
\ee
Using direct call of computing trigonometric functions i.e. the view of (\ref{scerror}),
\be
\begin{matrix}
\widehat{\textbf{G}}_k={\textbf{G}}_k+ \Delta{{\textbf{G}}_k}, \hspace{.1in}  |\Delta{{\textbf{G}}_k}| \leq \mu |{\textbf{G}}_k|,
\end{matrix}
\nonumber
\ee
Thus, overall
\small
\be
\begin{aligned}
\widehat{{\bf y}}&={\textbf{P}}^T_0\: {\textbf{P}}^T_1\left({\textbf{F}}_1 + \Delta{\textbf{F}_1}\right)\cdots{\textbf{P}}^T_{t-2}\left({\textbf{F}}_{t-2} + \Delta{\textbf{F}_{t-2}}\right)({\textbf{C}}_{t-1} + \Delta{\textbf{C}_{t-1}})
\\& \hspace{.3in}
({\textbf{G}}_{t-2}+{{\textbf{E}}_{t-2}})\cdots( {\textbf{G}}_{2}+{{\textbf{E}}_{2}})( {\textbf{G}}_{1}+{\textbf{E}}_{1})
( {\textbf{G}}_{0}+\Delta{{\textbf{G}}_{0}}) {\bf x},
\end{aligned}
%\label{1Es1}
\nonumber
\ee
%%%%%%%%%%%%%%%%%%%%%%%%%%%%
\be
\begin{matrix}
%|{\textbf{E}}_{t-1}| \leq (\mu +\gamma_3(1+\mu))|{\textbf{C}}_{t-1}| \leq \gamma_5|{\textbf{C}}_{t-1}|,\\
|{\textbf{E}}_{k}| \leq (\mu +\gamma_3(1+\mu))|{\textbf{G}}_{k}| \leq \gamma_5 |{\textbf{G}}_{k}|
\end{matrix}
\nonumber
\ee
%%%%%%%%%%%%%%%%%%%%%%%%%%%%%%
Hence
\be
\begin{matrix}
\begin{aligned}
\left | {\bf y} -\widehat{{\bf y}} \right | \leq  & \left[ (1+\gamma_2)^{t-1} (1+\gamma_3)(1+\gamma_5)^{t-2} -1\right] 
{\textbf{P}}^T_0\: {\textbf{P}}^T_1 \left| \textbf{F}_1\right|\cdots  
\\ 
& \hspace{.3in}{\textbf{P}}^T_{t-2} \left| \textbf{F}_{t-2}\right|\left| \textbf{C}_{t-1}\right| \left| \textbf{G}_{t-2}\right| \left| \textbf{G}_{t-3}\right| \cdots\left| \textbf{G}_{0}\right| \left|{\bf x}\right|
%\label{2Es1}
\end{aligned}
\end{matrix}
\nonumber
\ee
where 
\begin{eqnarray}
(1+\gamma_2)^{t-1}(1+\gamma_3) (1+\gamma_5)^{t-2} -1\leq (1+\gamma_2)^{t-1} (1+\gamma_5)^{t-1} -1
\nonumber \\
\leq (1+\gamma_7)^{t-1} -1 \leq \frac{\gamma_7(t-1)}{1-\gamma_7(t-1)}.
\nonumber
\end{eqnarray}
Since $\textbf{F}_k, \textbf{C}_{t-1}, \textbf{G}_k$ are orthogonal matrices, $\left \| \textbf{F}_k\right\|_2 = \left\|\textbf{C}_{t-1}\right\|_2=\left \| \textbf{G}_k \right\|_2=1$. By orthogonality of $C_{n}^{II}, \left \| {\bf y}\right\|_2 =\left \| {\bf x} \right\|_2$. Hence
\be
\left \| {\bf y}-\widehat{{\bf y}}\right \|_2 \leq \frac{\gamma_7(t-1)}{1-\gamma_7(t-1)} \left \| {\bf y} \right \|_2
%\label{3Es1}
\nonumber
\ee
\end{proof}
%%%%%%%%%%%%%%%%%%%%%%%%%%%%%%%%%%%%%%%%
%%%%%%%%%%%%%%%%%%%%%%%%%%%%%%%%%%%%%%%%%%5
%%%% Remark 6.15
\begin{corollary}
\label{CErrC2}
${\bf y}=C_{n}^{II}\:{\bf x}$ is forward and backward stable.
\end{corollary}
%%%%%%%%%%%%%%%%%%%
\begin{proof}
The above theorem says that radix 2 DCT-II yields a tiny forward error provided that ${\rm sin} \:\frac{r\pi}{4n}$ and ${\rm cos}\:\frac{r\pi}{4n}$ are computed stably. It immediately follows that the computation is backward stable because $\widehat{{\bf y}}={\bf y}+\Delta{{\bf y}}=C_{n}^{II}{\bf x} +\Delta{{\bf y}}$ implies $\widehat{{\bf y}}=C_{n}^{II}({\bf x}+\Delta{{\bf x}})$ with $\frac{\left\| \Delta{{\bf x}} \right\|_2}{\left\| {{\bf x}} \right\|_2}=\frac{\left\| \Delta{{\bf y}} \right\|_2}{\left\| {{\bf y}} \right\|_2}$. If we form ${\bf y}=C_{n}^{II}{\bf x}$ by using exact $C_{n}^{II}$, then $\left|{\bf y}- \widehat{{\bf y}} \right| \leq \gamma_{n}\:\left|C_{n}^{II}\right|\:\left|{\bf x}\right|$ so $\left\|{\bf y}- \widehat{{\bf y}} \right\|_2 \leq \gamma_{n}\:\left\|{\bf y}\right\|_2$. As $\mu$ is of order $u$, the $C_{n}^{II}$ has an error bound smaller than that for usual multiplication by the same factor as the reduction in complexity of the method, so DCT-II is perfectly stable. 
\end{proof}
%%%%%%%%%%%%%%%%%%%%%%%%
%%%%%%%%%%%%%%%%%%%
%%%%%%%%%%%%%%%%%
%%%%%%%%%%%%%%%%%%%%%% DCT IV
The error bound of computing recursive DCT-IV algorithm can be derived as follows.
%%%% Theorem 6.20
\bt \label{ThErrC4}
Let $\widehat{{\bf y}}=fl(C_n^{IV}{\bf x})$, where $n=2^t(t \geq 2)$, be computed using the algorithms $\bf{(\ref{algo:c4})}$, $\bf{(\ref{algo:c2})}$, and assume that (\ref{scerror}) holds, then
\be 
\frac{\left \| {\bf y}-\widehat{{\bf y}} \right \|_2}{\left \| {\bf y} \right \|_2}\leq \frac{\gamma_7t}{1-\gamma_7t}.
\label{Ec4}
\ee
\et
%%%%%%%%%%%%%%%%%%%%%%%%%%%%%%%%%%%%
%%%%%%%%%%%%%%%%%%%%%%%%%%%%%%%%%%%%%%%
%%%%%%%%%%%%%%%%%%%%%%%%%%%
\begin{proof}
Using the algorithms $\bf{(\ref{algo:c4})}$, $\bf{(\ref{algo:c2})}$, and the computed matrices $\widehat{\textbf{W}}_{k}$ (in terms of the computed $\widehat{d}_r$) for $k=0,1,\cdots,t-2$:
\small
\be
\begin{aligned}
\widehat{{\bf y}}&=fl\left({\textbf{P}}^T_0\:{\textbf{U}}_0 {\textbf{P}}^T_1{\textbf{U}}_1\cdots{\textbf{P}}^T_{t-2}{\textbf{U}}_{t-2}\:{\textbf{C}}_{t-1}\:\widehat{\textbf{W}}_{t-2}\cdots \widehat{\textbf{W}}_{1}{\widehat{\textbf{W}}_{0}}\: {\bf x}\right)\\
&={\textbf{P}}^T_0\left({\textbf{U}}_0 + \Delta{\textbf{U}_0}\right)\cdots{\textbf{P}}^T_{t-2}\left({\textbf{U}}_{t-2} + \Delta{\textbf{U}_{t-2}}\right)({\textbf{C}}_{t-1} + \Delta{{\textbf{C}}_{t-1}})
\\& \hspace{.3in}
(\widehat{\textbf{W}}_{t-2}+\Delta{\widehat{\textbf{W}}_{t-2}})\cdots( \widehat{\textbf{W}}_{1}+\Delta{\widehat{\textbf{W}}_{1}})( {\widehat{\textbf{W}}_{0}}+\Delta{\widehat{\textbf{W}}_{0}}) {\bf x}
\end{aligned}
%\label{1Es1}
\nonumber
\ee
Each $\textbf{U}_k$ is formed containing a combination of matrices ${I}_{\frac{n}{2^k}}$ and $U_{\frac{n}{2^k}}$ except $\textbf{U}_0=U_n$. Using the fact that each row in $\textbf{U}_k$ has at most two non-zero entries with mostly ones per row:
\be 
\begin{matrix}
\left | \Delta{{\textbf{U}}_k} \right | \leq {\gamma}_{2}\left | \textbf{U}_k \right |\:\:\:{\rm for}\:\:\:k=0,1,\cdots,t-2
\end{matrix}
%\label{rEs1}
\nonumber
\ee
Also each ${\textbf{W}}_k$ is formed containing a combination of matrices ${H}_{\frac{n}{2^k}}$ and $R_{\frac{n}{2^k}}$ except $\textbf{W}_0={R}_{n}$. Using the fact that each row in ${\textbf{W}_k}$ has at most two non-zero entries per row:
\be 
\begin{matrix}
\left | \Delta{\widehat{\textbf{W}}_k} \right | \leq {\gamma}_{3}\:\left | \widehat{\textbf{W}}_k \right |,\\
{\rm for}\:\:\:k=0,1\cdots,t-2
\end{matrix}
%\label{tEs1}
\nonumber
\ee
${\textbf{C}}_{t-1}$ is a block diagonal matrix containing $C_{2}^{II}$ and $C_{2}^{IV}$ hence
\be 
\left | \Delta{{\textbf{C}}_{t-1}} \right | \leq {\gamma}_{3}\:\left | {\textbf{C}}_{t-1} \right |
%\label{ssEs1}
\nonumber
\ee
Using direct call of computing trigonometric functions i.e. the view of (\ref{scerror}),
\be
\begin{matrix}
\widehat{\textbf{W}}_k={\textbf{W}}_k+ \Delta{{\textbf{W}}_k}, \hspace{.1in}  |\Delta{{\textbf{W}}_k}| \leq \mu |{\textbf{W}}_k|,
\end{matrix}
\nonumber
\ee
Thus, overall
\small
\be
\begin{aligned}
\widehat{{\bf y}}&={\textbf{P}}^T_0\: \left({\textbf{U}}_0 + \Delta{\textbf{U}_0}\right)\cdots{\textbf{P}}^T_{t-2}\left({\textbf{U}}_{t-2} + \Delta{\textbf{U}_{t-2}}\right)({\textbf{C}}_{t-1} + \Delta{\textbf{C}_{t-1}})
\\& \hspace{.3in}
({\textbf{W}}_{t-2}+{{\textbf{E}}_{t-2}})\cdots({\textbf{W}}_{1}+{\textbf{E}}_{1})
( {\textbf{W}}_{0}+{\textbf{E}}_{0}) {\bf x},
\end{aligned}
%\label{1Es1}
\nonumber
\ee
%%%%%%%%%%%%%%%%%%%%%%%%%%%%
\be
\begin{matrix}
%|{\textbf{E}}_{t-1}| \leq (\mu +\gamma_3(1+\mu))|{\textbf{C}}_{t-1}| \leq \gamma_5|{\textbf{C}}_{t-1}|,\\
|{\textbf{E}}_{k}| \leq (\mu +\gamma_3(1+\mu))|{\textbf{W}}_{k}| \leq \gamma_5 |{\textbf{W}}_{k}|
\end{matrix}
\nonumber
\ee
%%%%%%%%%%%%%%%%%%%%%%%%%%%%%%
Hence
\be
\begin{matrix}
\begin{aligned}
\left | {\bf y} -\widehat{{\bf y}} \right | \leq  & \left[ (1+\gamma_2)^{t-1} (1+\gamma_3)(1+\gamma_5)^{t-1} -1\right] 
{\textbf{P}}^T_0\: \left| \textbf{U}_0\right|\cdots  
\\ 
& \hspace{.3in} {\textbf{P}}^T_{t-2} \left| \textbf{U}_{t-2}\right|\left| \textbf{C}_{t-1}\right| \left| \textbf{W}_{t-2}\right|  \cdots\left| \textbf{W}_{1}\right|\left| \textbf{W}_{0}\right| \left|{\bf x}\right|
%\label{2Es1}
\end{aligned}
\end{matrix}
\nonumber
\ee
where 
\begin{eqnarray}
(1+\gamma_2)^{t-1}(1+\gamma_3) (1+\gamma_5)^{t-1} -1 &\leq& (1+\gamma_3) (1+\gamma_7)^{t-1} -1
\nonumber \\
& \leq& (1+\gamma_7)^{t} -1\leq \frac{\gamma_7t}{1-\gamma_7t}.
\nonumber 
\end{eqnarray}
Since $\textbf{U}_k, \textbf{C}_{t-1}, \textbf{W}_k$ are orthogonal matrices, $\left \| \textbf{U}_k\right\|_2 = \left\|\textbf{C}_{t-1}\right\|_2=\left \| \textbf{W}_k \right\|_2=1$. By orthogonality of $C_{n}^{IV}, \left \| {\bf y}\right\|_2 =\left \| {\bf x} \right\|_2$. Hence
\be
\left \| {\bf y}-\widehat{{\bf y}}\right \|_2 \leq \frac{\gamma_7t}{1-\gamma_7t} \left \| {\bf y} \right \|_2
%\label{3Es1}
\nonumber
\ee
\end{proof}
%%%%%%%%%%%%%%%%%%%%%%%%%%%%%%%%%%%%%%%%

%%%%%%%%%%%%%%%%%%%%%%%%%%%%%%%%%%%
%%%% Remark 6.21
\begin{corollary}
\label{StabC4}
${\bf y}=C_{n}^{IV}\:{\bf x}$ is forward and backward stable.
\end{corollary}
%%%%%%%%%%%%%%%%%%%%%%%%%%%%%
%%%%%%%%%%%%%%%%%%%
\begin{proof}
The above theorem says that radix 2 DCT-IV yields a tiny forward error provided that ${\rm sin} \:\frac{r\pi}{4n}$ and ${\rm cos}\:\frac{r\pi}{4n}$ are computed stably. It immediately follows that the computation is backward stable because $\widehat{{\bf y}}={\bf y}+\Delta{{\bf y}}=C_{n}^{IV}{\bf x} +\Delta{{\bf y}}$ implies $\widehat{{\bf y}}=C_{n}^{IV}({\bf x}+\Delta{{\bf x}})$ with $\frac{\left\| \Delta{{\bf x}} \right\|_2}{\left\| {{\bf x}} \right\|_2}=\frac{\left\| \Delta{{\bf y}} \right\|_2}{\left\| {{\bf y}} \right\|_2}$. If we form ${\bf y}=C_{n}^{IV}{\bf x}$ by using exact $C_{n}^{IV}$, then $\left|{\bf y}- \widehat{{\bf y}} \right| \leq \gamma_{n}\:\left|C_{n}^{IV}\right|\:\left|{\bf x}\right|$ so $\left\|{\bf y}- \widehat{{\bf y}} \right\|_2 \leq \gamma_{n}\:\left\|{\bf y}\right\|_2$. As $\mu$ is of order $u$, the $C_{n}^{IV}$ has an error bound smaller than that for usual multiplication by the same factor as the reduction in complexity of the method, so DCT-IV is perfectly stable. 
\end{proof}

%%%%%%%%%%%%%%%%%%%%%%%%%% Error of DCT-III

\begin{corollary}
\label{StabC2}
Let $\widehat{{\bf y}}=fl(C_{n}^{III} {\bf x})$, where $n=2^t(t \geq 2)$, be computed using the algorithms $\bf{(\ref{algo:c3})}$, $\bf{(\ref{algo:c4})}$, $\bf{(\ref{algo:c2})}$, and assume that (\ref{scerror}) holds, then
\be 
\frac{\left \| {\bf y}-\widehat{{\bf y}} \right \|_2}{\left \| {\bf y} \right \|_2}\leq \frac{\gamma_7(t-1)}{1-\gamma_7(t-1)}.
\label{Ec3}
\ee
\end{corollary}
%%%%%%%%%%%%%%%%%%%%%%%%%%%%%%%%%%%%%
\begin{corollary}
\label{StabC3}
${\bf y}=C_{n}^{III}\:{\bf x}$ is forward and backward stable.
\end{corollary}

%%%%%%%%%%%%%%%%%%%%%%%%%%%%%%%%%%%%%%%%%%%%%%%
%%%%%%%%%%%%%%%%%%%%%%%%%%%%%%%%
%%%%%%%%%%%%% Error for DCT I

Finally, the error bound for computing DCT-I algorithm, which runs recursively with DCT II-IV algorithms, can be derived as follows.

% Theorem 6.5
\bt \label{ThErrC1}
Let $\widehat{{\bf y}}=fl(C_{n+1}^{I} {\bf x})$, where $n=2^t(t \geq 2)$, be computed using the algorithms $\bf{(\ref{algo:c1})}$, $\bf{(\ref{algo:c3})}$, $\bf{(\ref{algo:c4})}$, $\bf{(\ref{algo:c2})}$, and assume that (\ref{scerror}) holds. Then
\be 
\frac{\left \| {\bf y}-\widehat{{\bf y}} \right \|_2}{\left \| {\bf y} \right \|_2}\leq \frac{\gamma_7t}{1-\gamma_7t}.
\label{Ec1}
\ee
\et
%%%%%%%%%%%%%%%%%%%%%%%%%%%%%
\begin{proof}
Using the algorithms $\bf{(\ref{algo:c1})}$, $\bf{(\ref{algo:c3})}$, $\bf{(\ref{algo:c4})}$, $\bf{(\ref{algo:c2})}$, and the computed matrices $\widehat{\textbf{B}}_{k}$ (in terms of the computed $\widehat{d}_r$) for $k=2,3,\cdots,t-2$:
\small
\be
\begin{aligned}
\widehat{{\bf y}}&=fl\left({\textbf{A}}_0\: {\textbf{A}}_1\cdots{\textbf{A}}_{t-2}\:{\textbf{C}}_{t-1}\:\widehat{\textbf{B}}_{t-2}\cdots \widehat{\textbf{B}}_{2}{\textbf{B}}_{1}{\textbf{B}}_{0}\: {\bf x}\right)\\
&=\left({\textbf{A}}_0 + \Delta{\textbf{A}_0}\right)\cdots\left({\textbf{A}}_{t-2} + \Delta{\textbf{A}_{t-2}}\right)({\textbf{C}}_{t-1} + \Delta{{\textbf{C}}_{t-1}})
\\& \hspace{.3in}
(\widehat{\textbf{B}}_{t-2}+\Delta{\widehat{\textbf{B}}_{t-2}})\cdots( \widehat{\textbf{B}}_{2}+\Delta{\widehat{\textbf{B}}_{2}})
( {\textbf{B}}_{1}+\Delta{{\textbf{B}}_{1}})( {\textbf{B}}_{0}+\Delta{{\textbf{B}}_{0}}) {\bf x}
\end{aligned}
%\label{1Es1}
\nonumber
\ee
Each $\textbf{A}_k$ is formed containing a combination of matrices ${P}_{\frac{n}{2^k}+1}^T$, ${P}_{\frac{n}{2^k}}^T$, ${H}_{\frac{n}{2^k}}^T$ and $U_{\frac{n}{2^k}}$ except $\textbf{A}_0={P}_{n+1}^T$ and $\textbf{A}_1={\rm blkdiag}\left({P}_{\frac{n}{2}+1}^T,H^T_{\frac{n}{2}}\right)$. Using the fact that each row in $\textbf{A}_k$ has at most two non-zero entries with mostly ones per row:
\be 
\begin{matrix}
\left | \Delta{{\textbf{A}}_0} \right |=0,
& &
\left | \Delta{{\textbf{A}}_k} \right | \leq {\gamma}_{2}\left | \textbf{A}_k \right |\:\:\:{\rm for}\:\:\:k=1,2,\cdots,t-2
\end{matrix}
%\label{rEs1}
\nonumber
\ee
Also each ${\textbf{B}}_k$ is formed containing a combination of matrices $\breve{H}_{\frac{n}{2^k}+1}$, ${H}_{\frac{n}{2^k}}$, ${P}_{\frac{n}{2^k}}$ and $R_{\frac{n}{2^k}}$ except $\textbf{B}_0=\breve{H}_{n+1}$ and $\textbf{B}_1={\rm blkdiag}\left(\breve{H}_{\frac{n}{2}+1}, {P}_{\frac{n}{2}} \right)$. Using the fact that each row in ${\textbf{B}_k}$ has at most two non-zero entries per row:
\be 
\begin{matrix}
\left | \Delta{{\textbf{B}}_0} \right | \leq {\gamma}_{2}\:\left | {\textbf{B}}_0 \right |,\:\:\: 
\left | \Delta{{\textbf{B}}_1} \right | \leq {\gamma}_{2}\:\left | {\textbf{B}}_1 \right |,\:\:\:
\left | \Delta{\widehat{\textbf{B}}_k} \right | \leq {\gamma}_{3}\:\left | \widehat{\textbf{B}}_k \right |,\\
{\rm for}\:\:\:k=2,3,\cdots,t-2
\end{matrix}
%\label{tEs1}
\nonumber
\ee
${\textbf{C}}_{t-1}$ is a block diagonal matrix containing $C_{1}^{I}$, $C_{2}^{II}$, $C_{2}^{III}$ and $C_{2}^{IV}$ hence
\be 
\left | \Delta{{\textbf{C}}_{t-1}} \right | \leq {\gamma}_{3}\:\left | {\textbf{C}}_{t-1} \right |
%\label{ssEs1}
\nonumber
\ee
Using direct call of computing trigonometric functions i.e. the view of (\ref{scerror}),
\be
\begin{matrix}
\widehat{\textbf{B}}_k={\textbf{B}}_k+ \Delta{{\textbf{B}}_k}, \hspace{.1in}  |\Delta{{\textbf{B}}_k}| \leq \mu |{\textbf{B}}_k|,
\end{matrix}
\nonumber
\ee
Thus, overall
\small
\be
\begin{aligned}
\widehat{{\bf y}}&=\left({\textbf{A}}_0 + \Delta{\textbf{A}_0}\right)\cdots\left({\textbf{A}}_{t-2} + \Delta{\textbf{A}_{t-2}}\right)({\textbf{C}}_{t-1} + \Delta{\textbf{C}_{t-1}})
\\& \hspace{.3in}
({\textbf{B}}_{t-2}+{{\textbf{E}}_{t-2}})\cdots( {\textbf{B}}_{2}+{{\textbf{E}}_{2}})( {\textbf{B}}_{1}+\Delta{{\textbf{B}}_{1}})
( {\textbf{B}}_{0}+\Delta{{\textbf{B}}_{0}}) {\bf x},
\end{aligned}
%\label{1Es1}
\nonumber
\ee
%%%%%%%%%%%%%%%%%%%%%%%%%%%%
\be
\begin{matrix}
%|{\textbf{E}}_{t-1}| \leq (\mu +\gamma_3(1+\mu))|{\textbf{C}}_{t-1}| \leq \gamma_5|{\textbf{C}}_{t-1}|,\\
|{\textbf{E}}_{k}| \leq (\mu +\gamma_3(1+\mu))|{\textbf{B}}_{k}| \leq \gamma_5 |{\textbf{B}}_{k}|
\end{matrix}
\nonumber
\ee
%%%%%%%%%%%%%%%%%%%%%%%%%%%%%%
Hence
\be
\begin{matrix}
\begin{aligned}
\left | {\bf y} -\widehat{{\bf y}} \right | \leq  & \left[ (1+\gamma_2)^{t} (1+\gamma_3)(1+\gamma_5)^{t-3} -1\right] 
 \left| \textbf{A}_0\right|\left| \textbf{A}_1\right|\cdots  \left| \textbf{A}_{t-2}\right| 
\\ 
& \left| \textbf{C}_{t-1}\right| \left| \textbf{B}_{t-2}\right| \left| \textbf{B}_{t-3}\right| \cdots\left| \textbf{B}_{0}\right| \left|{\bf x}\right|
%\label{2Es1}
\end{aligned}
\end{matrix}
\nonumber
\ee
where 
\begin{eqnarray}
(1+\gamma_2)^{t}(1+\gamma_3) (1+\gamma_5)^{t-3} -1 &\leq& (1+\gamma_2)^{t} (1+\gamma_5)^{t-2} -1
\nonumber \\
&\leq& (1+\gamma_7)^{t} -1
\nonumber \\
&\leq& \frac{\gamma_7t}{1-\gamma_7t}.
\nonumber
\end{eqnarray}
Since $\textbf{A}_k, \textbf{C}_{t-1}, \textbf{B}_k$ are orthogonal matrices, $\left \| \textbf{A}_k\right\|_2 = \left\|\textbf{C}_{t-1}\right\|_2=\left \| \textbf{B}_k \right\|_2=1$. By orthogonality of $C_{n+1}^{I}, \left \| {\bf y}\right\|_2 =\left \| {\bf x} \right\|_2$. Hence
\be
\left \| {\bf y}-\widehat{{\bf y}}\right \|_2 \leq \frac{\gamma_7t}{1-\gamma_7t} \left \| {\bf y} \right \|_2
%\label{3Es1}
\nonumber
\ee
\end{proof}

%%%% Remark 6.6
\begin{corollary}
\label{StabC1}
${\bf y}=C_{n+1}^{I}\:{\bf x}$ is forward and backward stable.
\end{corollary}

%%%%%%%%%%%%%%%%%%%%
%%%%%%%%%%%%%%%%%%%%%%%%
\begin{proof}
The above theorem says that radix 2 DCT-I yields a tiny forward error provided that ${\rm sin} \:\frac{r\pi}{4n}$ and ${\rm cos}\:\frac{r\pi}{4n}$ are computed stably. It immediately follows that the computation is backward stable because $\widehat{{\bf y}}={\bf y}+\Delta{{\bf y}}=C_{n+1}^{I}{\bf x} +\Delta{{\bf y}}$ implies $\widehat{{\bf y}}=C_{n+1}^{I}({\bf x}+\Delta{{\bf x}})$ with $\frac{\left\| \Delta{{\bf x}} \right\|_2}{\left\| {{\bf x}} \right\|_2}=\frac{\left\| \Delta{{\bf y}} \right\|_2}{\left\| {{\bf y}} \right\|_2}$. If we form ${\bf y}=C_{n+1}^{I}{\bf x}$ by using exact $C_{n+1}^{I}$, then $\left|{\bf y}- \widehat{{\bf y}} \right| \leq \gamma_{n+1}\:\left|C_{n+1}^{I}\right|\:\left|{\bf x}\right|$ so $\left\|{\bf y}- \widehat{{\bf y}} \right\|_2 \leq \gamma_{n+1}\:\left\|{\bf y}\right\|_2$. As $\mu$ is of order $u$, the $C_{n+1}^{I}$ has an error bound smaller than that for usual multiplication by the same factor as the reduction in complexity of the method, so DCT-I is perfectly stable. 
\end{proof}

%%%%%%%%%%%%%%%%%%%%%%%%%%%%%%%
%%%%%%%%%%%%%%%%%%%%%%%%%%%%%
%%%%%%%%%%%%%%%%%%%%%%%%%%%%%%%%%%%%%%%
%%%%%%%%%%%%%%%%%%%%%%%%%%%%%%%%
%%%%%%%%%%%%%%%%%%%%%%%%%%%%%%%%%%
%%%%%%%%%%%%%%%%%%%%%%%%%%%%
%%%%%%%%%%%%%%%%%%%%%%% Image on these
\section{Image compression results based on DCT algorithms}
\label{sec:IMC}
Discretized images can be considered as matrices. To compress such images one can apply the quantization technique. In this section we use the quantization technique with the help of recursive DCT-II and DCT-IV algorithms to compress the Lena image of size $512 \times 512$ pixels. At first, the image is discretized into $8 \times 8$, $16 \times 16$, and $32 \times 32$ transfer blocks. Next, using the recursive DCT-II and DCT-IV algorithms, 2D-DCTs are computed for each block. The DCT-II and DCT-IV coefficients are then quantized by transforming absence of 93.75$\%$ of the DCT coefficients (93.75$\%$ of DCT-II and DCT-IV coefficients in each transfer block are set to zero). In each block, the inverse 2D DCT-II and DCT-IV coefficients are computed. Finally, putting each block back together into a single image leads to Figures \ref{fig:LenafigC2} and \ref{fig:LenafigC4}.    
\newline\newline
Figure \ref{fig:LenafigC2} shows images with discarded coefficients (except the top left $6.25\%$ in each transfer block) in each transfer block, after applying DCT-II algorithm, and then running recursively with the DCT-IV algorithm.  
%%%%%%%%%%%%%%%%%%%%%
\begin{figure}[h]
\centering
\begin{subfigure}{0.21\textwidth}
 \includegraphics[width=\textwidth]{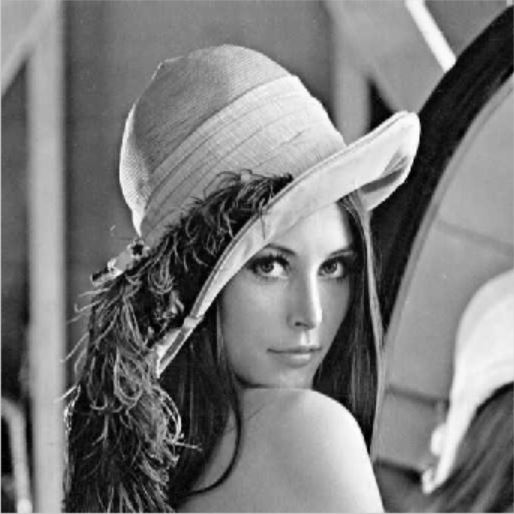}
 \caption{}
 \label{fig:origiL2}
\end{subfigure}
\begin{subfigure}{0.21\textwidth}
\includegraphics[width=\textwidth]{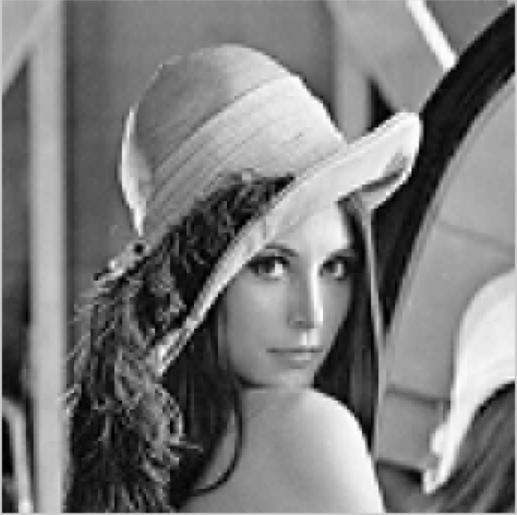}
\caption{}
 \label{fig:new8imgC2}
 \end{subfigure}

%%%%%%%%%%%%%%%%%%%%%%%%%%%%%
%%%%%%%%%%%%%%%%%%%%%%%%%%%%
 \begin{subfigure}{0.21\textwidth}
  \includegraphics[width=\textwidth]{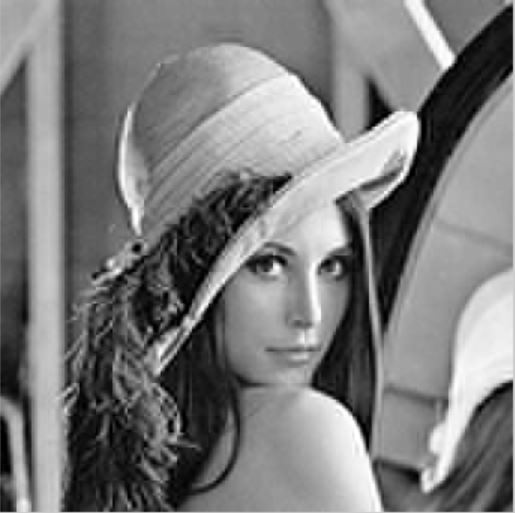}
  \caption{}
 \label{fig:new16imgC2}
 \end{subfigure}
 \begin{subfigure}{0.21\textwidth}
 \includegraphics[width=\textwidth]{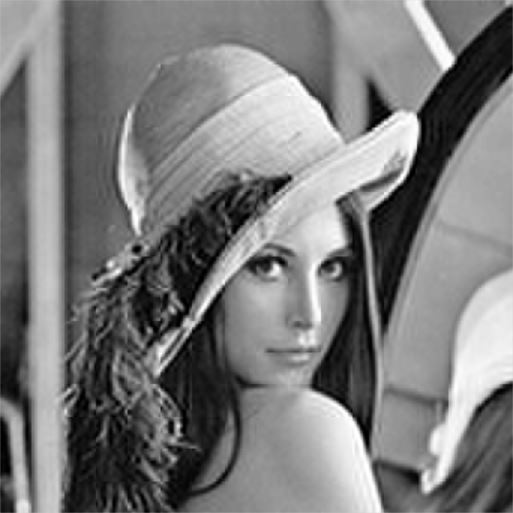}
 \caption{}
 \label{fig:new32imgC2}
 \end{subfigure}
\caption{(\ref{fig:origiL2}) Original Lena Image \:(\ref{fig:new8imgC2}) Reconstructed image with 93.75$\%$ discarded DCT-II coefficients in each $8 \times 8$ transfer block \: (\ref{fig:new16imgC2}) Reconstructed image with 93.75$\%$ discarded DCT-II coefficients in each $16 \times 16$ transfer block \:(\ref{fig:new32imgC2}) Reconstructed image with 93.75$\%$ discarded DCT-II coefficients in each $32 \times 32$ transfer block}
\label{fig:LenafigC2}
\end{figure}   
\\\\
%%%%%%%%%%%%%%%%%%%%%%%%%%%%%%%%%%%%%%%%%%%
Figure \ref{fig:LenafigC4} shows images with discarded coefficients (except the top left $6.25\%$ in each transfer block) in each transfer block after applying DCT-IV algorithm and then running recursively with the DCT-II algorithm.    
\\
\begin{figure}[h]
\centering
\begin{subfigure}{0.21\textwidth}
\includegraphics[width=\textwidth]{Lena-Ori}
\caption{}
\label{fig:origiL4}
\end{subfigure}
\begin{subfigure}{0.21\textwidth}
\includegraphics[width=\textwidth]{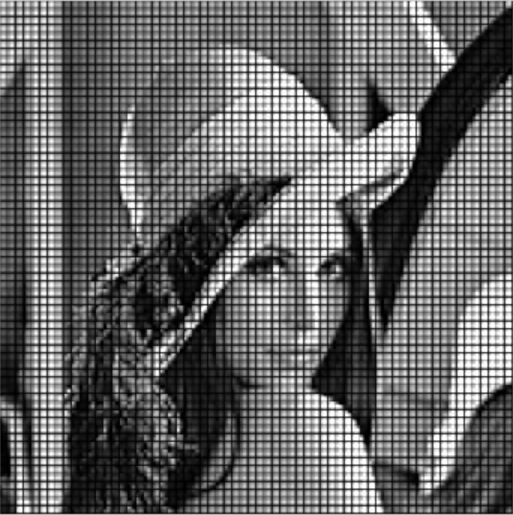}
\caption{}
\label{fig:new8imgC4}
\end{subfigure}

%%%%%%%%%%%%%%%%%%%%%%%%%%%%%
%%%%%%%%%%%%%%%%%%%%%%%%%%%%
 \begin{subfigure}{0.21\textwidth}
                 \includegraphics[width=\textwidth]{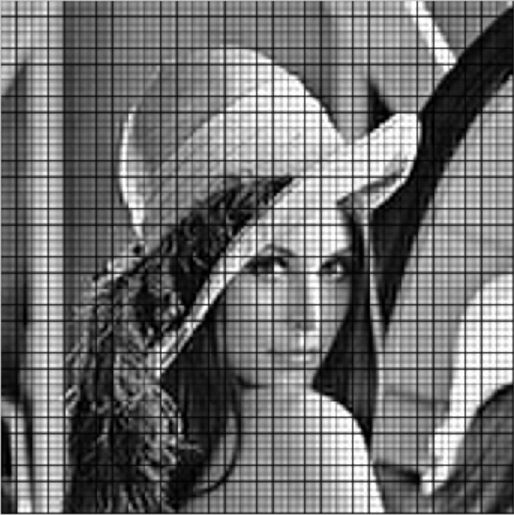}
               \caption{}
                 \label{fig:new16imgC4}
         \end{subfigure}
   \begin{subfigure}{0.21\textwidth}
                 \includegraphics[width=\textwidth]{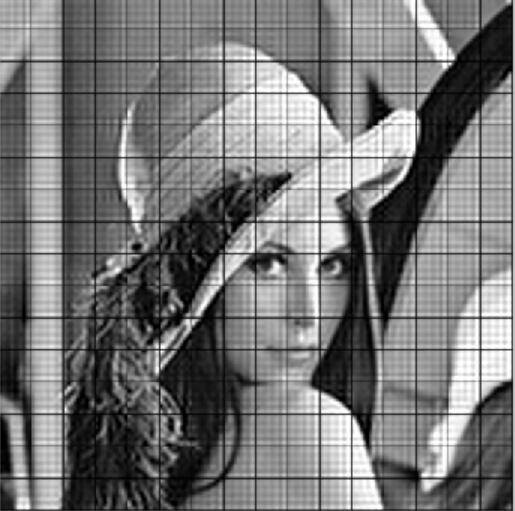}
               \caption{}
                 \label{fig:new32imgC4}
         \end{subfigure}
 \caption{(\ref{fig:origiL4}) Original Lena Image \:(\ref{fig:new8imgC4}) Reconstructed image with 93.75$\%$ discarded DCT-IV coefficients in each $8 \times 8$ transfer block \: (\ref{fig:new16imgC4}) Reconstructed image with 93.75$\%$ discarded DCT-IV coefficients in each $16 \times 16$ transfer block \:(\ref{fig:new32imgC4}) Reconstructed image with 93.75$\%$ discarded DCT-IV coefficients in each $32 \times 32$ transfer block}
\label{fig:LenafigC4}
\end{figure}   
%%%%%%%%%%%%%%%%%%%%%%%%%%%%%%%%%%%%%%
\\
Comparing to Figures \ref{fig:LenafigC2} and \ref{fig:LenafigC4}, the image reconstruction results corresponding to DCT-II algorithm are better than that of the DCT-IV algorithm. Though the quality of reconstructed images in Figures \ref{fig:LenafigC2} and \ref{fig:LenafigC4} are somewhat lost, those images are clearly recognizable even though $93.75\%$ of the DCT-II and DCT-IV coefficients are discarded in each transfer block.

%%            Signal Flow graphs for DST I-IV
%%%%%%%%%%%%%%%%%%%%%%%%%%%%%%%%%%%%%%%%%%%%%%%%%%
%%%%%%%%%%%%%%%%%%%%%%%%%%%%%%%%%%%%%%%%%%%%%%%%%%
%%%%%%%%%%%%%%%%%%%%%%%%%%%%%%%%%%%%%%%%%%%%%%%%%%%%%
%%%%%%%%%%%%%%%%%%%%%%%%%%%%%%%%%%%%%%%%%%%%%
%%%%%%%%%%%%%%%%%%%%%%%%%%%%%%%%%%%%%%%%%%%%%%%%%%%
%%%%%%%%%%%%%%%%%%%%%%%%%%%%%%%%%%%%%%%%%%%%%%%%%%%
\section{Signal flow graphs for DCT algorithms}
\label{sec:SFG}
Signal flow graphs commonly represent the realization of systems such as electronic devices in electrical engineering, control theory, system engineering, theoretical computer science, etc. Simply put, the objective is to build a device to implement or realize an algorithm, using devices that implement the algebraic operations used in these recursive algorithms. These building blocks are shown next in Figure \ref{figBB}.
%%%%%%%%%%%%%%%%%%
\begin{figure}[h]
\center
\includegraphics[width=\linewidth,height=.7in]{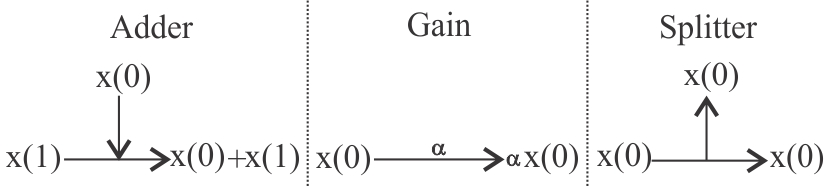}
\caption{Signal flow graphs building blocks}
\label{figBB} 
\end{figure}
%%%%%%%%%%%%%%%%%%%%%%%%%%%%%%%%%
This section presents signal flow graphs for 9-point DCT-I and 8-point DCT II-IV algorithms via Figures \ref{figC1}, \ref{figC2}, \ref{figC3}, and \ref{figC4}.  As shown in the flow graphs, in each graph signal flows from the left to the right. These signal flow graphs are corresponding to the decimation-in-frequency algorithms. However one can convert these decimation-in-frequency DCT algorithms into decimation-in-time DCT algorithms. In each Figure (\ref{figC1}, \ref{figC2}, \ref{figC3}, and \ref{figC4}), $\epsilon:=\frac{1}{\sqrt{2}}$, $C_{i,j}:=\cos\frac{i \pi}{2^j}$, and $S_{i,j}=\sin \frac{i\pi}{2^j}$.
%%%%%%%%%%%%%%%%%%%%%%%%%%%%%%%%%%%%%%
%%%%%%%%%%%%%%%%%%%%%%%%%%%%%%%%%%%%%%%%%%
\begin{figure}[h]
\center
\includegraphics[width=\linewidth,height=2in]{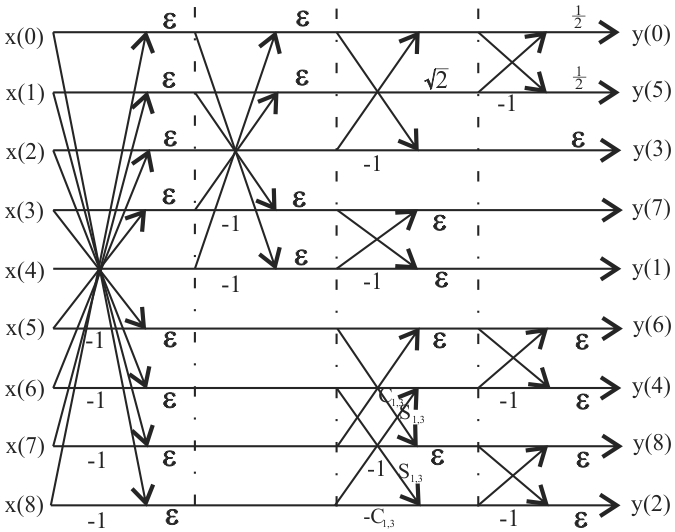}
\caption{Flow graph for 9-point DCT-I algorithm}
\label{figC1} 
\end{figure}
%%%%%%%%%%%%%%%%%%%%%%%%%%%%%%%%%%%%%%%%%
\begin{figure}[h]
\center
\includegraphics[width=\linewidth,height=2in]{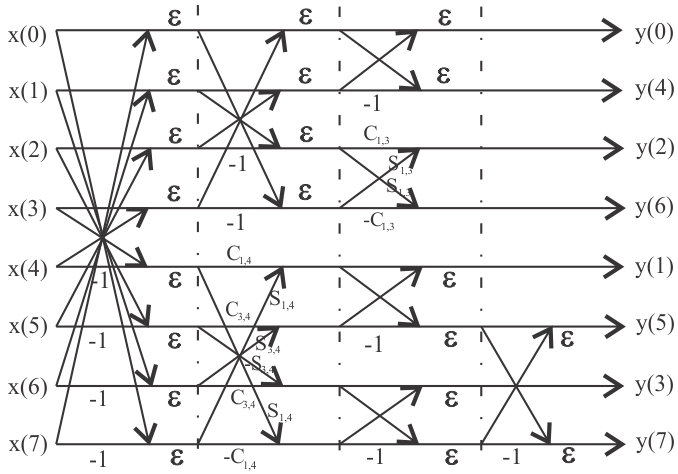}
\caption{Flow graph for 8-point DCT-II algorithm}
\label{figC2} 
\end{figure}
%%%%%%%%%%%%%%%%%%%%%%%%%%%%%%%%%%%%%%%%%%%%%%
\begin{figure}[h]
\center
\includegraphics[width=\linewidth,height=2in]{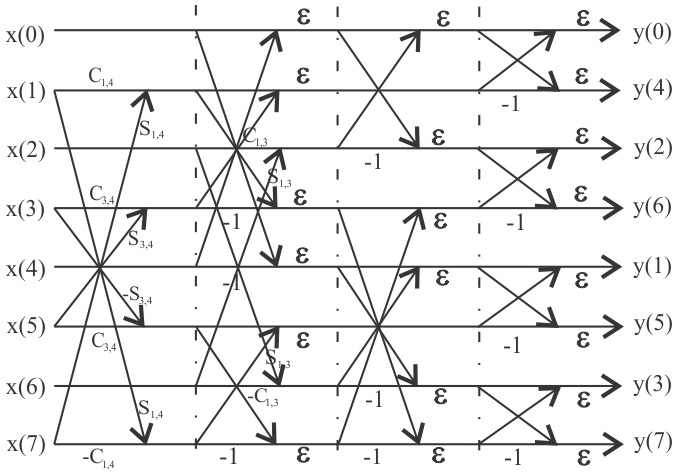}
\caption{Flow graph for 8-point DCT-III algorithm}
\label{figC3} 
\end{figure}
%\newline
%%%%%%%%%%%%%%%%%%%%%%%%%%%%%%%%%%%%%%%%%%%%%%%%%%%
\begin{figure}[h]
\center
\includegraphics[width=\linewidth,height=2in]{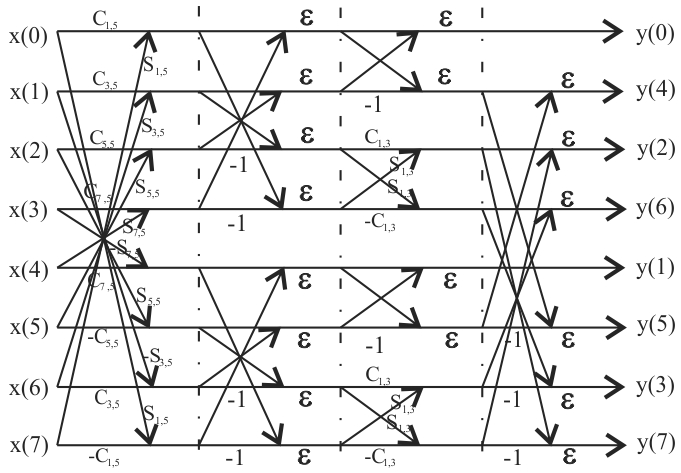}
\caption{Flow graph for 8-point DCT-IV algorithm}
\label{figC4} 
\end{figure}
%%%%%%%%%%%%%%%%%%%%%%%%%%%%%%%%%%%%%%%%%%%%
%%%%%%%%%%%%%%%%%%%%%%%%%%%%%%%%%%%%%%%%%%%%%%%%
%%%%%%%%%%%%%%%%%%%%%%%%%
%%%%%%%%%%%%%%%%%%%%%%%%%%%%%
%%%%%%%%%%%%%%%%%%%%%%%%%%%%%%%%%%%%%%%%%%%%
As shown in the Figures \ref{figC2}, \ref{figC3}, and \ref{figC4}, the input signals are in order: ${\bf x} = \{x(0), x(1), \cdots, x(7)\}$ and output signals are in bit-reversed order: ${\bf y} = \{y(0), y(4), y(2), y(6), y(1), y(5), y(3), y(7)\}$. In bit-reversed order, each output index is represented as a binary number and the indices' bits are reversed. Say for 8-point DCT II, the sequential order of the input indices' bits is $\{000, 001, 010, 011, 100, 101, 110, 111\}$ then reversing these input signal bits yields $\{000, 100, 010, 110, 001, 101, 011,111\}$ which is the output signal.

%%%%%%%%%%%%%%%%%%%%%%%%
%%%%%%%%%%%%% Conclusion
%%%%%%%%%%%%%%%%%%%%%%%%%%%%%%%%%%%%
%%%%%%%%%%%%%%%%%%%%%%%%%%%%%%%%%
%%%%%%%%%%%%%%%%%%%%%%%%%%%%%%%%%
\section{Conclusion}
This paper provided stable, completely recursive, radix-2 DCT-I and DCT-III algorithms having sparse, orthogonal and rotation/rotation-reflection matrices, defined solely via DCT I-IV algorithms. The arithmetic cost and error bounds of computing DCT I-IV algorithms are addressed. Using the recursive DCT-II and DCT-IV algorithms with the absence of $93.75\%$ coefficients in each transfer block in 2D DCT-II and DCT-IV, one can reconstruct $512 \times 512$ images without seriously affecting the quality. Signal flow graphs are presented for these solely based orthogonal factorization of DCT I-IV in decimation-of-frequency.

%%%%%%%%%%%%%%%%%%%%%%%%%%%%%%%%%%%%%%%

%%%%%%%%%%%%%%%%%%%%%%%%%%%%%%%%%%%%%%%%
%% BIBLIOGRAPHY
%%%%%%%%%%%%%%%%%%%%%%%%%%%%%%%%%%%%%%%%

\normalsize

\end{document}